\documentclass[12pt]{extarticle}
\usepackage{amsfonts,amsthm,amsmath,latexsym,amssymb}
\usepackage{bbm}
\usepackage{tikz} 
\usepackage[a4paper, total={7in, 9in}]{geometry} 
\usepackage[margin=1cm]{caption}
\usepackage{float}
\usepackage{enumerate} 

\usepackage{algorithm,algorithmic}
\usepackage{graphics}
\usepackage{color}
\usepackage[noadjust]{cite}
\usepackage{authblk}
\usepackage{setspace}
\usepackage{comment}
\newtheorem{theorem}{Theorem}[section]
\newtheorem{lemma}[theorem]{Lemma}
\newtheorem{claim}[theorem]{Claim}
\newtheorem{proposition}[theorem]{Proposition}
\newtheorem{corollary}[theorem]{Corollary}

\theoremstyle{definition}
\newtheorem{definition}[theorem]{Definition}
\theoremstyle{remark}
\newtheorem{remark}[theorem]{Remark}
\numberwithin{equation}{section}

\newcommand{\cH}{\mathcal{H}}

\newcommand{\cE}{\mathcal{E}}

\newcommand{\cN}{\mathcal{N}}

\newcommand{\cP}{\mathcal{P}}

\newcommand{\bP}{\mathbb{P}}

\newcommand{\bE}{\mathbb{E}}
\newcommand{\Q}{\mathbb{Q}}

\newcommand{\Z}{\mathbb{Z}}

\newcommand{\R}{\mathbb{R}}

\newcommand{\la}{\lambda }

\newcommand{\de}{\delta }

\newcommand{\si}{\sigma }

\newcommand{\ga}{\gamma }
\newcommand{\Ga}{\Gamma }

\newcommand{\ones}{\mathbbm{1}}
\newcommand{\one}{\mathbf{1}}
\newcommand{\Mod}{\operatorname{Mod}}
\newcommand{\Dom}{\operatorname{Dom}}

\newcommand{\Adm}{\operatorname{Adm}}
\newcommand{\Ext}{\operatorname{Ext}}
\newcommand{\MEO}{\operatorname{MEO}}

\newcommand{\co}{\operatorname{co}}
\newcommand{\bi}{\begin{itemize}}
\newcommand{\ei}{\end{itemize}}
\newcommand{\Gahat}{\widehat{\Ga}}
\newcommand{\Gatil}{\widetilde{\Ga}}
\newcommand{\BL}{\operatorname{BL}}
\newcommand{\lbr}{\left\{ }
\newcommand{\rbr}{\right\} }
\newcommand{\argmax}{\operatorname{argmax}}

\newcommand{\coni}{\operatorname{coni}}

\date{}
\begin{document}
\title{\bf Fulkerson duality for modulus of spanning trees and  partitions\thanks{This material is based upon work supported by the National Science Foundation under Grant n. 2154032.}}
\author[1]{Huy Truong}
\author[1]{Pietro Poggi-Corradini}
\affil[1]{\small Dept. of Mathematics, Kansas State University, Manhattan, KS 66506, USA.}

\maketitle 


\begin{abstract}
One of the main properties of modulus on graphs is Fulkerson duality. In this paper, we study Fulkerson duality for spanning tree modulus.  We introduce a new notion of Beurling partition, and we identify two important ones, which correspond to the notion of strength and maximum denseness of an arbitrary graph. These special partitions, also give rise to two deflation processes that reveal a hierarchical structure for general graphs. While Fulkerson duality for spanning tree families can be deduced from a well-known result in combinatorics due to Chopra, we give an alternative approach based on a result of Nash-Williams and Tutte. Finally, we introduce the weighted variant of spanning tree modulus.
\end{abstract}
\noindent {\bf Keywords:} Spanning trees, feasible partitions, Fulkerson duality, modulus, strength.

\vspace{0.1in}

\noindent {\bf 2020 Mathematics Subject Classification:} 90C27 (Primary) ; 05C85  (Secondary).
\section{Introduction}

In this paper, we consider graphs $G=(V,E)$ that are finite, undirected, and connected multigraphs with vertex set $V$ and edge set $E$. We allow $G$ to have multiedges, but not self-loops. Often, we will simply call $G$ a ``connected graph" for short. In some sections, we consider weighted graphs $G=(V,E,\si)$, meaning that each edge $e$ has a  weight $0<\si(e)<\infty$ assigned to it. In other sections, we consider unweighted graphs, meaning that every edge $e$ has weight $\si(e)\equiv 1$.

Modulus on graphs has been developed and investigated extensively in recent years
\cite{modulus, pietrominimal, pietroblocking, pietrofairest, polynomial}. 
Modulus of many different families has been studied and used in applications. For example, the modulus of families of paths on graphs recovers path-related quantities such as minimum cut, effective resistance and shortest path \cite{modulus}. Modulus of the family of spanning trees was introduced in \cite{pietrofairest}. Several applications and algorithms have been developed for spanning tree modulus  \cite{pietrosecure, polynomial}. 
Fulkerson duality for modulus exhibits a dual family of a given family. For instance, when looking at the family of all paths connecting two nodes $s$ and $t$, the dual family in this case is known to be the family of minimal $st$-cuts. In the case of the family of spanning trees of a graph, the Fulkerson dual is known to be indexed by a set of feasible partitions, as can be deduced from a result of Chopra \cite{chopraon} (see Section \ref{sec-feasible-spanning}).

It is widely recognized that the relationship between the strength of $G=(V,E,\si)$ (see Definition \ref{def:strength-pb}) and the tree packing problem (see (\ref{dualr})) can be derived from a result of Nash-Williams and Tutte \cite{edge-disjoint,on} (see Theorem \ref{TutteN}). While the existence of this derivation is mentioned in the literature, a formal proof is yet to be identified. Thus, to our knowledge, the validity of this derivation is not self-evident. In this paper, we establish this relationship. In \cite{frank2003decomposing}, the authors give a generalization of the theorem of Nash-Williams and Tutte to hypergraphs. However, the following example shows that an equivalent formulation for this relationship does not hold for hypergraphs. Consider the following hypergraph $H =( V =\lbr v_1,v_2,v_3 \rbr, E = \lbr e_1,e_2\rbr)$ and $e_1 = \lbr v_1,v_2,v_3 \rbr$, $e_2 = \lbr v_1,v_2,v_3 \rbr$. When $\si = (1,2)$, we have that its strength 
 equals to $3/2$ and 
the hypertree packing value is $1$.
 Based on this observation, we conclude that the use of the Nash-Williams and Tutte's theorem is specific to graphs, and cannot be readily generalized to the case of hypergraphs. Consequently, we recognize the significance of our proof specifically in the context of graphs.

As a stepping stone, we utilize this result to deduce Fulkerson duality for spanning trees, an approach distinct from Chopra's method in \cite{chopraon}. 
Moreover, in \cite{ba2023}, the authors show that the equivalent Fulkerson duality formulation does not hold for hypergraphs.

In this paper, we also thoroughly analyze spanning tree modulus. Below is a summary of our contributions:

\begin{itemize}
	\item In Section \ref{sec:pairs}, we introduce the notion of Fulkerson dual pairs.
	\item In Section \ref{sec:beu}, we introduce the concept of Beurling partitions. 
	
	\item In Section \ref{sec4} and Section \ref{sec5}, we give two important Beurling partitions, they represent the strength and the maximum denseness of arbitrary graphs.
	
	\item In Section \ref{sec6}, we provide several useful equivalent characterizations of homogeneous graphs.
	
	\item In section \ref{sec:deflation}, we give two deflation processes for general graphs.

	\item In Section \ref{sec:firsth}, we establish a connection between the strength of graphs and the 1-modulus of the spanning tree family.

	\item In Section \ref{sec:ful}, we reprove the Fulkerson duality for the spanning tree family $\Ga$.
	
	\item In Section \ref{sec:weighted2.2}, we generalize our results to spanning tree modulus with weights.

\end{itemize}


\section{Preliminaries}

\subsection{Modulus on graphs and Fulkerson duality for modulus}
Let $G= (V,E,\si)$ be a weighted graph with edge weights $\si \in \R^E_{>0}$. Let $\Ga$ be a finite family of objects that are subsets of the 
edge set $E$. For each object $\ga \in \Ga$, we associate a function 
\[\cN(\ga,\cdot)^T: E\rightarrow \R_{\geq 0}, \]
where $\cN(\ga,\cdot)^T$ is an {\it usage vector} in $\R^E_{\geq 0}$ which measures the {\it usage of each edge} $e$ by $\ga$. In other words, $\Ga$ is associated with a $|\Ga| \times |E|$ {\it usage matrix} $\cN$.  To avoid trivial situations, we will assume from now on that the family $\Ga$ is non-empty and that each object $\ga\in\Ga$ uses at least some edges a positive and finite amount.

Let $\rho$ be a {\it density}  in $ \R^E_{\geq 0}$ which consists in values on the edges of the graph, the value of $\rho(e)$ represents the {\it cost} of using the edge $e \in E$. Given an object $\ga \in \Ga$, we define 
\[\ell_{\rho}(\ga) := \sum\limits_{e \in E} \cN(\ga,e)\rho(e)= (\cN\rho)(\ga),\] which represents the {\it total usage cost} of $\ga$ with respect to the density $\rho$.
A density $\rho \in \R^E_{\geq 0 }$ is  {\it admissible} for $\Ga$, if
\[ \ell_{\rho}(\ga) \geq 1, \quad \forall \ga\in \Ga,\]
or equivalently, if 
\[\ell_{\rho}(\Ga) := \inf\limits_{\ga\in \Ga} \ell_{\rho}(\Ga) \geq 1 .\]
In matrix notations, $\rho$ is admissible if \[\cN\rho \geq \one,\]
where $\one$ is the column vector of all ones and the inequality  holds elementwise.  The {\it admissible set} $\Adm(\Ga)$ of $\Ga$ is the set of all admissible densities for $\Ga$,
\begin{equation}\label{eq:adm-set}
\Adm(\Ga) := \left\{ \rho \in \R^E_{\geq 0 }: \cN\rho \geq \one \right\}. 
\end{equation} 
Fix $1 \leq p < \infty$, the {\it energy} of the density  $\rho$ is defined as follows
\[ \cE_{p,\si}(\rho):=\sum\limits_{e \in E}\si(e)\rho(e)^p.\] 
When $p =\infty $,
\[\cE_{\infty,\si}(\rho):= \max\limits_{e\in E} \left\{ \si (e)\rho(e)\right\}.\]
The {\it $p$-modulus} of a family of objects $\Ga$ is

\[\Mod_{p,\si}(\Ga):= \inf\limits_{\rho \in \Adm(\Ga)} \cE_{p,\si}(\rho).\]
Equivalently, $p$-modulus of $\Ga$ corresponds to the following optimization problem,

\begin{equation} \label{modp}
\begin{array}{ll}
\underset{\rho \in \R^{E}}{\text{minimize}}    &\cE_{p,\si}(\rho) \\
\text{subject to } &\sum\limits_{e \in E} \cN(\ga,e)\rho(e) \geq 1, \quad \forall \ga \in \Ga; \\
						 &\rho \geq 0.  	 
\end{array}
\end{equation}
When $\si$ is the vector of all ones, we write $\Mod_{p}(\Ga) := \Mod_{p,\si}(\Ga)$.

In this section, we start by recalling some important notions. Let $K$ be a closed convex set in $\R^E_{\geq 0}$ such that $ \varnothing \subsetneq K \subsetneq \R^E_{\geq 0}$ and $K$ is {\it recessive}, in the sense that $K +\R^E_{\geq 0} = K$. We define the {\it blocker} $\BL(K)$ of $K$ as follows, \[\BL(K) = \lbr \eta \in \R^E_{\geq 0} : \eta^{T}\rho \geq 1, \forall \rho \in K \rbr. \] 
Let $\Ga$ be a family of objects. We will often identify the family $\Ga$ with the set of its usage vectors $\left\{ \cN(\ga,\cdot )^T: \ga \in \Ga \right\}$ in $\R^E_{\geq 0}$, so we can write $\Ga \subset \R^E_{\geq 0 }$. Then, the {\it dominant} of $\Ga$ is
\[ \Dom(\Ga):= \co(\Ga) + \R^E_{\geq 0},\]
where $\co(\Ga)$ denotes the convex hull of $\Ga$ in $\R^E$.

Also, note that the admissible set  $\Adm(\Ga)$ defined in (\ref{eq:adm-set}) is closed, convex in $ \R^E_{\geq 0}$ and recessive. 
Next, we recall the {\it Fulkerson blocker family} of $\Ga$. 
\begin{definition} Let $\Ga$ be a family of objects.
	The {\it Fulkerson blocker family} $\widehat{\Ga}$ of $\Ga$ is defined as the set of all the extreme points of $\Adm(\Ga)$.
	\[ \widehat{\Ga} := \Ext\left(\Adm(\Ga)\right) \subset  \R^E_{\geq 0}.\]
\end{definition}
Fulkerson blocker duality \cite{fulkersonblocking} states that
\[\Dom(\widehat{\Ga})= \Adm(\Ga) = \BL(\Adm(\Gahat)),\] 
\[\Dom(\Ga)= \Adm(\widehat{\Ga}) = \BL\left(\Adm(\Ga)\right).\]
Since $\widehat{\Ga}$ is a set of vectors in $\R^E_{\geq 0}$, $\widehat{\Ga}$  has its own Fulkerson blocker family, and
\begin{equation}\label{eq:gahathat-ga}
	\widehat{\widehat{\Ga}}  \subset \Ga .
\end{equation}

When  $1<p < \infty$, let $q:=p/(p-1)$ be the
Hölder conjugate exponent of $p$. For any set of weights $\si \in \R^E_{>0}$,  define the dual set of weights $\widetilde{\si}$ as $\widetilde{\si}(e):=\si(e)^{-\frac{q}{p}}$ for all $e\in E$. Let $\Gahat$ be a Fulkerson dual family of $\Ga$. Fulkerson duality for modulus  \cite[Theorem 3.7]{pietroblocking} states that
\begin{equation}
	\Mod_{p,\si}(\Ga)^{\frac{1}{p}}\Mod_{q,\widetilde{\si}}(\widehat{\Ga})^{\frac{1}{q}}=1.
\end{equation}
Moreover, the optimal $\rho^*$ of $\Mod_{p,\si}(\Ga) $ and the optimal $\eta^*$ of $\Mod_{q,\widetilde{\si}}(\widehat{\Ga})$ always exist, are unique, and are related as follows,
\[\eta^{\ast}(e) = \frac{\si(e)\rho^{\ast}(e)^{p-1}}{\Mod_{p,\si}(\Ga)} \quad \forall e\in E.\]
In the special case when $p=2$, we have
\begin{equation*}
	\Mod_{2,\si}(\Ga)\Mod_{2,\si^{-1}}(\widehat{\Ga})=1 \qquad\text{and}\qquad   \displaystyle \eta^{\ast}(e) = \frac{\si(e)}{\Mod_{2,\si}(\Ga)}\rho^{\ast}(e) \quad \forall  e\in E.
\end{equation*}

\subsection{The minimum expected overlap problem}
Let $ \Ga = \Ga_G$ be the set of all spanning trees of a connected multigraph $G = (V,E).$ Let $\cP(\Ga)$ be the set of all probability mass functions (pmf) on $\Ga$. Given a pmf $\mu \in \cP(\Ga)$, let $\underline{\ga}$ and $\underline{\ga'}$ be two independent random spanning trees, identically distributed with law $\mu$. The cardinality of the overlap between $\underline{\ga}$ and $\underline{\ga'}$, is $|\underline{\ga} \cap \underline{\ga'}|$ and is a random variable whose expectation is denoted by  $\bE_\mu|\underline{\ga}\cap\underline{\ga'}|$. Then, the {\it minimum expected overlap} ($\MEO$) problem is the following problem,

\begin{equation} \label{meo}
{\renewcommand{\arraystretch}{1.8}
\begin{array}{ll}
\text{minimize}    &\bE_\mu|\underline{\ga}\cap\underline{\ga'}| \\
\text{subject to } & \mu \in \cP(\Ga). 	 
\end{array}
}
\end{equation}
The optimal laws $\mu^*$ for the $\MEO$ problem are not in general unique. However, the corresponding edge usage probabilities $\eta^*(e) = \bP_{\mu^*}(e \in \underline{\ga})$ do not depend on the optimal pmf $\mu^*$. Moreover, they satisfy
$\sum\limits_{e \in E} \eta^*(e) = |V|-1$.

\begin{theorem}[\cite{pietrominimal}] \label{meomod}
Let $G = (V,E)$ be a connected multigraph, let $\Ga = \Ga_G$ be the spanning tree family with usage vectors given by the indicator functions, and let $\widehat{\Ga}$ be the Fulkerson blocker family of $\Ga$. Then $\rho \in  \R^E_{\geq 0}$, $\eta \in  \R^E_{\geq 0}$ and $\mu \in \cP(\Ga)$ are optimal respectively for $\Mod_2(\Ga)$, $\Mod_2(\widehat{\Ga})$ and $\MEO(\Ga)$ if and only if the following conditions are satisfied.
\begin{itemize}
\item[(i)] $\rho \in \Adm(\Ga), \hspace{2pt} \eta = \cN^T\mu,$
\item[(ii)] $\eta(e) = \frac{\rho(e)}{\Mod_2(\Ga)} \quad \forall e \in E,$
\item[(iii)] $\mu(\ga)(1-\ell_\rho(\ga)) = 0 \quad \forall \ga \in \Ga.$
\end{itemize}
In particular, 
\begin{align*}
\MEO(\Ga) = \Mod_2(\Ga)^{-1} = \Mod_2(\widehat{\Ga}).
\end{align*}
\end{theorem}

 A spanning tree $\ga \in \Ga$ is called a {\it fair tree} if there exists an optimal pmf $\mu^*$ such that $\mu^*(\ga)>0.$ The set of all fair trees of $G$ is denoted by $\Ga^f$.
Let $\rho^*$ and $\eta^*$ be the unique optimal densities for $\Mod_2(\Ga)$ and $\Mod_2(\widetilde{\Ga})$, respectively.  
\begin{definition}\label{def:homogeneous}
The graph $G$ is said to be {\it homogeneous} (with respect to spanning tree modulus) if $\eta^*$ is constant, or equivalently, $\rho^*$ is constant. 
\end{definition}
By \cite[Corollary 4.4]{pietrofairest}, $G$ is homogeneous if and only if $\eta_{ hom} \equiv \frac{|V|-1}{|E|} \in \Adm(	\widehat{\Ga})$.

For a large graph $G$, the family $\Ga_G$ of all spanning trees is typically very large. Each tree gives rise to a constraint in the spanning tree modulus problem, hence it would be computationally infeasible for a standard quadratic program solver to solve the modulus problem with all those constraints. However, an algorithm for spanning tree modulus was described in \cite{pietrofairest} which allows one to effectively solve the problem within a given tolerance. An algorithm in exact arithmetic for spanning tree modulus was given in \cite{polynomial}.

\subsection{Feasible partitions and spanning trees}\label{sec-feasible-spanning}
Let $G=(V,E)$ be a connected graph. Let $\Ga =\Ga_G$ be the spanning tree family. A partition $P= \left\{ V_1,V_2,...,V_{k_P} \right\}$ of the node set $V$ with $k_P \geq 2$  is said to be {\it feasible} if each $V_i$ induces a connected subgraph $G(V_i)$ of $G$.   The set of all edges in $G$ that connect vertices that belong to different $V_i$ is called the {\it cut set} of $P$ and denoted by $E_P$. The {\it weight} of $P$ is defined as follows,
 \begin{equation}\label{eq:weight-partition}
 w(P) = \frac{|E_P|}{k_P-1}.
 \end{equation}

Given a feasible partition $P= \left\{ V_1,V_2,...,V_{k_P} \right\}$, we construct the {\it shrunk graph} $G_P$ of the partition $P$ by identifying all the nodes in $V_i$ as a single node $v_i$ and removing all resulting self-loops. The graph $G_P$  has $k_P$ nodes and  the edge set $E_P$. Because $G$ is connected, the shrunk graph $G_P$ is connected. Furthermore, every spanning tree $\ga$ of $G$ must connect all the subsets $V_i$. Therefore, the restriction of $\ga$ onto $E_P$ forms a connected spanning subgraph of $G_P$. The notion of feasible partitions was introduced in  \cite{chopraon}. \begin{definition}\label{def:strength-pb}
	Let $G= (V,E,\si)$ be a weighted connected graph with edge weights $\si \in \R^E_{>0}$. Let  $\Phi$ be the set of all feasible partitions. If $P = \left\{ V_1,V_2,...,V_{k_P} \right\}$ is a feasible partition of the node set $V$, let $E_P$ be the cut set, and denote $\si(E_P) := \sum\limits_{e \in E_P} \si(e)$. Then,  we define the {\it strength} of the graph $G$ as follows,
	\begin{equation} \label{strength}
		S_{\si}(G) := \min\limits_{P \in \Phi} \frac{\si(E_P)}{k_P-1}.
	\end{equation}
	Any feasible partition $P$ satisfying $S_{\si}(G)= \frac{\si(E_P)}{k_P-1} $ is called a  {\it critical partition}. In the case when $G$ is unweighted, we write $S(G):=S_{\si \equiv \one} (G).$
\end{definition}
The notion of strength for graphs was suggested by Gusfield \cite{gusfieldtree} as a measure of 
network invulnerability. The second notion we want to recall is the maximum denseness of graphs. Maximum denseness is also known as the fractional arboricity of the graph $G$ which measures the maximum edge density of subgraphs of $G$. The fractional arboricity has been studied for many years and has many applications. 
Next, we define the {\it denseness} $\theta_{\si}(G)$ of the graph $G =(V,E,\si)$:
\begin{equation}\label{denseness}
	\theta_{\si}(G) : = \frac{\si(E)}{|V|-1}.
\end{equation}
When $G$ is unweighted, we write $\theta(G):=\theta_{\si \equiv \one} (G).$
The function $\theta$ gives a sense of how dense or sparse G is. There are several other functions that tell how dense a graph is, for example, the graph density  $\frac{2|E|}{|V|(|V|-1)}$ or the edge-based graph density $\frac{|E|}{|V|}$. Since the edge-based density is half the average degree of the graph $G$, the denseness is also closely related to  the average degree. Now, the maximum denseness of weighted graphs is defined as follows.
\begin{definition} Let $G=(V,E\si)$ be a connected graph. Let $\cH$ be the set of all vertex-induced connected subgraphs of $G$ that contain at least one edge. We define the {\it maximum denseness} $D(G)$ of the graph $G$ as follows,
	\begin{equation}\label{eq:denseness}
		D_{\si}(G) :=  \max\limits_{H' \in \cH} \theta_{\si}(H').
	\end{equation}
	We call the problem of finding $D_{\si}(G)$ the {\it maximum denseness  problem}. Any $H\in \cH$ satisfies $\theta_{\si}(H) = \max\limits_{H' \in \cH} \theta_{\si}(H')$ is called a {\it maximum denseness subgraph} of $G$. When $G$ is unweighted, we write $D(G):=D_{\si \equiv \one} (G).$
\end{definition}

Recall that $\Ga$ is the spanning tree family. The polyhedron  $\Dom(\Ga)$ has been studied thoroughly. In \cite{chopraon}, Chopra shows that 
\begin{align}\label{eq:chopraa}
\Dom(\Ga) = \left\{ \eta \in \R^E_{\geq 0} : \sum\limits_{e\in E_P} \eta(e) \geq k_P-1 \text{ for all feasible partitions } P \right\}.
\end{align}
 We restate the result  (\ref{eq:chopraa}) in the language of modulus. Let $\Phi$ be the family of all feasible partitions $P$ of $G$ with the usage vectors:

\begin{equation}\label{usage2}
\widetilde{\cN}(P,\cdot)^T=\frac{1}{k_P-1}\ones_{E_P}.
\end{equation}
Then, the admissible set of $\Phi$ is:
\begin{equation*}
\Adm(\Phi) = \left\{ \eta \in \R^E_{\geq 0} :  \widetilde{\cN} \eta \geq \one \right\}.
\end{equation*}
where $\one$ is the column vector of all ones.
 Notice that the set of usage vectors of $\Phi$ is 

 \[ \left\{ \frac{1}{k_P-1}\mathbbm{1}_{E_P}  : P \text{ is a feasible partition}\right\}, \]
 and each usage vector in the set is indexed by the corresponding feasible partition. We will often identify $\Phi$ with its family of usage vectors in $\R^E_{\geq 0}$.
\begin{theorem}[\cite{chopraon}]\label{theo:chopra} Let $G=(V,E)$ be a connected graph. Let $\Ga =\Ga_G$ be the spanning tree family of $G$.
Let $\Phi$ be the family of all feasible partitions $P$ of $G$ with usage vectors given as in (\ref{usage2}).  Then, $\Phi$ is a Fulkerson dual family of $\Ga$.
\end{theorem}

 Theorem \ref{theo:chopra} does not give the Fulkerson blocker family $\Gahat$ of the spanning tree family $\Ga$, but by Proposition \ref{theo:smallestF}, we have that 
\begin{align}\label{fact:1}
\Gahat \subset \Phi.
\end{align}

\begin{definition}\label{def:vertex-biconnected} A graph $G=(V,E)$ is said to be {\it vertex-biconnected}, if it has at least two vertices, is connected and the removal of any vertex along with all incident edges of that vertex does not disconnect the graph.
\end{definition}

In \cite{chopraon}, Chopra shows that an inequality in (\ref{eq:chopraa}) defines a facet of $\Dom(\Ga)$ if and only if the shrunk graph $G_P$ is vertex-biconnected \cite[Theorem 3.2]{chopraon}, thus he gives a minimal inequality description of $\Dom(\Ga)$.  That leads to a precise understanding of the Fulkerson blocker family $\widehat{\Ga}$ of $\Ga$. Let $\Theta \subset \Phi$ be the family of all feasible partitions $P$ whose shrunk graph $G_P$ is vertex-biconnected, with usage given as in (\ref{usage2}). Similar to the family $\Phi$, we also identify $\Ga, \widehat{\Ga}$, and $\Theta$ with the corresponding families of usage vectors. Then, \cite[Theorem 3.2]{chopraon} implies that 
\begin{align}\label{fact:0}
\Gahat = \Theta.
\end{align}
The proof of (\ref{eq:chopraa}) in \cite{chopraon} is based on a minimum spanning tree algorithm.


\section{Fulkerson dual pairs and $1$-modulus}\label{sec:pairs}
Let $\Ga$ be a family of objects. In this section, we introduce the notion of {\it Fulkerson dual pairs}. 

\begin{definition}
	Let $\Ga$ and $\widetilde{\Ga}$ be two sets of vectors in $\R^E_{\geq 0}$. We say that $\Ga$ and $\widetilde{\Ga}$ are a {\it Fulkerson dual pair} (or $\Gatil$ is a {\it Fulkerson dual family} of $\Ga$) if \[\Adm(\Gatil) = \BL(\Adm(\Ga)).\] 
\end{definition}
\begin{remark}
	If $\Gatil$ is a Fulkerson dual family of $\Ga$, then $\Ga$ is also a Fulkerson dual family of $\Gatil$ since \[\BL(\Adm(\Gatil)) = \BL(\BL(\Adm(\Ga))) = \Adm(\Ga).\]
\end{remark}
\begin{proposition}\label{theo:smallestF}
	Let $\Ga$ be a set of vectors in $\R^E_{\geq 0}$. Let $\Gahat$ be the Fulkerson blocker family of $\Ga$ and $\Gatil$ be a Fulkerson dual family of $\Ga$. Then, \[\Gahat \subset \Gatil.\] In other words, $\Gahat$ is the smallest Fulkerson dual family of $\Ga$. 
\end{proposition}
\begin{proof}
	We have \[\Adm(\Gatil) = \BL(\Adm(\Ga)) = \Adm(\Gahat) = \Dom(\Ga)\] and \[\Dom(\Gatil) = \BL(\Adm(\Gatil)) = \BL(\Adm(\Gahat)) = \Adm(\Ga).\]
	Therefore, we obtain that 
	\begin{equation}\label{eq:gatilde-admga}
		\Gatil \subset \Dom(\Gatil) = \Adm(\Ga),
	\end{equation}
	and
	\begin{equation}\label{eq:gahat-domgatilde}
		\Gahat \subset \Adm(\Ga) = \Dom(\Gatil).
	\end{equation}
	Let $\widehat{\ga} \in \Gahat$, we want to show that $\widehat{\ga} \in \Gatil$. By (\ref{eq:gahat-domgatilde}), we have \[\widehat{\ga} = \sum\limits_{i\in I} \mu_i \widetilde{\ga}_i + z,\] where $\widetilde{\ga}_i \in \Gatil$ and $\mu_i\ge 0$ for all $i\in I$, $\sum\limits_{i\in I}\mu_i =1$, and $z \in \R^E_{\geq 0}.$ 
	Then, we rewrite $\widehat{\ga}$ as follows,
	\[ \widehat{\ga} = \frac{1}{2}\left( \sum\limits_{i\in I} \mu_i \widetilde{\ga}_i \right) + \frac{1}{2}\left( \sum\limits_{i\in I} \mu_i \widetilde{\ga}_i +2z \right).\]
	Because $\widehat{\ga}$ is an extreme point of $\Adm(\Ga)$ and by (\ref{eq:gatilde-admga}), we obtain
	\[
	\widehat{\ga} = \sum\limits_{i\in I} \mu_i \widetilde{\ga}_i= \sum\limits_{i\in I} \mu_i \widetilde{\ga}_i+2z.
	\]
	This implies that $z=0 \in \R^E_{\geq 0}$ and $\widehat{\ga} = \sum\limits_{i\in I} \mu_i \widetilde{\ga}_i$. Using the fact that $\widehat{\ga}$ is an extreme point of $\Adm(\Ga)$ one more time, we obtain $\widetilde{\ga}_i = \widehat{\ga}$ for all $i \in I$. Therefore, $\widehat{\ga} \in \Gatil$. In conclusion, $\Gahat \subset \Gatil.$
\end{proof}
As we have seen, $\Ga$ and $\Gahat$ are a Fulkerson dual pair and $\Ga$ is a Fulkerson dual family of $\Gahat$. A natural question to ask is under what conditions $\Ga$ is also the Fulkerson blocker family of $\Gahat$, namely, when do we have equality in (\ref{eq:gahathat-ga}).

	An element  $\ga \in \Ga$ is said to be  {\it essential} if  it represents a facet of $\Adm(\Ga)$ and {\it inessential} otherwise.  Note that $\ga$ is inessential if and only if the usage vector of $\ga$ is greater than or equal to a convex combination of different other usage vectors of $\Ga$. If every $\ga \in \Ga$ is essential, we say that  $\Ga$ is {\it proper}.
	
	In the special case, when all usage vectors of $\Ga$ belong to $\left\{ 0,1 \right\} ^E$. Then, $\Ga$ is proper if and only if the support set $\lbrace e \in E: \cN(\ga,e) \neq 0 \rbrace$ of any  usage vector of $\Ga$ does not contain the support set of any other usage vector of $\Ga$. Such a family is called a {\it clutter} in the combinatorics literature. For example, the spanning tree family of an undirected connected graph with usage vectors given by the indicator functions is a clutter.
	
	An important property of proper families is that if $\Ga$ is proper, then so is $\widehat{\Ga}$ and $\widehat{\widehat{\Ga}} =\Ga $, see \cite{fulkersonanti}.  In this case, $\Ga$ is the Fulkerson blocker family of $\Gahat$.

For some family of objects $\Gamma$, sometimes it is difficult to find the Fulkerson blocker family $\Gahat$, and it is easier to obtain a Fulkerson dual family $\Gatil$. Because we are interested in modulus, it will be enough to work with a dual family $\Gatil$, since what matters is the admissible set $\Adm(\Gatil)$.

Next,  we recall the 1-modulus problem $\Mod_{1,\si}(\Ga)$:

\begin{equation} \label{prob:mod1}
	\begin{array}{ll}
		\underset{\rho \in \R^{E}}{\text{minimize}}    &\si^T\rho \\
		\text{subject to } &\sum\limits_{e \in E} \cN(\ga,e)\rho(e) \geq 1, \quad \forall \ga \in \Ga; \\
		&\rho \geq 0.  	 
	\end{array}
\end{equation}
The dual problem of $\Mod_{1,\si}(\Ga)$ is
\begin{equation}  \label{dualr}
	\begin{array}{ll}
		\underset{\lambda \in \R^{\Ga}}{\text{maximize}}    &\lambda^T\mathbf{1} \\
		\text{subject to } &\sum\limits_{e \in \ga} \lambda(\ga) \leq \sigma(e), \quad \forall e \in E; \\
		&\lambda \geq 0, 	 
	\end{array}
\end{equation}
where $\one$ is the column vector of all ones. When $\Ga$ is the spanning tree family with usage vectors given by the indicator functions, the problem (\ref{dualr}) is known as the {\it tree packing} problem for graphs.

 For any set of weights $\sigma \in \mathbb{R}^E_{> 0}$, we have
\begin{align}\label{mod:mm}
	\Mod_{1,\sigma}(\Gamma)=\min\limits_{\widehat{\ga} \in \widehat{\Ga}}\sigma^{\top}\widehat{\ga}.
\end{align}
This follows because, when $p =1$, $1$-modulus is a linear program. Therefore, at least one optimal density $\rho^{\ast}$ in this case must occur at an extreme point of $\Adm(\Ga)$, see \cite[Theorem 2.8]{bertsimas}. Moreover, we have the following useful characterization of $\widehat{\Ga}$.
\begin{lemma}\label{lem:vertex}
	For every $\widehat{\gamma} \in \widehat{\Gamma}$, there is a choice of $\sigma \in \mathbb{R}_{>0}^E$ such that $\widehat{\gamma}$ is the unique solution for both $\Mod_{1,\sigma}(\Gamma)$ and $\min\limits_{\ga \in \widehat{\Ga}}\sigma^{\top}\ga$.
\end{lemma}
This is a modification of \cite[Theorem 2.3]{bertsimas}.
For the convenience of the reader, we recall the proof of \cite[Theorem 2.3]{bertsimas} and show the needed changes.
\begin{proof}[Proof of Lemma \ref{lem:vertex}]
	We recall that the polyhedron $\Adm(\Ga)$ is defined as 
	\[\Adm(\Ga) = \left\{ \rho \in \R^E:  a_\ga^T\rho \geq b_\ga \text{ for all } \ga \in \Ga, a_e^T \rho \geq b_e  \text{ for all } e \in E \right\}, \]
	where $a_\ga^T$ is the usage vector $\cN(\ga,\cdot)$ and $b_\ga=1$, while  $a_e^T$ is the indicator $\de_{e}$ and $b_e=0$. 
	
	Let  $\widehat{\gamma}$ be an extreme point of $\Adm(\Ga)$. By \cite[Theorem 2.3]{bertsimas}, there exists $|E|$ constraints of $\Adm(\Ga)$ denoted by $\left\{ a_i^T\rho \geq b_i : i =1,2,\dots,|E| \right\} $ such that for $ i =1,2,\dots,|E| $, 
	\begin{equation}\label{eq:aibi-equal}
		a_i^T\widehat{\gamma} = b_i,
	\end{equation} 
	and  $|E|$ vectors $ a_i$ are linearly independent. Here, each $(a_i,b_i)$ may equal $(a_\ga,b_\ga)$ for some $\ga\in\Ga$ or $(a_e,b_e)$ for some $e\in E$.
	
	Let $\si := \sum\limits_{i =1}^{|E|} a_i$. For any $\rho \in \Adm(\Ga)$, we have  $ a_i^T \rho \geq b_i$, for $ i =1,2,\dots,|E|$. Hence,
	\begin{align}
		\si^T\rho  =  \left(\ \sum\limits_{i =1}^{|E|} a_i^T \right)\rho &\geq  \sum\limits_{i =1}^{|E|} b_i
		\label{ex11} \\
		& =  \left(\ \sum\limits_{i =1}^{|E|} a_i^T \right)\widehat{\gamma} & \text{(by (\ref{eq:aibi-equal}))} \notag \\
		& =  \si^T\widehat{\gamma}. \notag
	\end{align}
	That means $\widehat{\gamma}$ is an optimal solution for $\Mod_{1,\sigma}(\Gamma)$. Furthermore, equality holds in (\ref{ex11}) if and only if  $ a_i^T \rho = b_i$ for $ i =1,2,\dots,|E|$. Since the vectors $ a_i$, for $i =1,2,\dots,|E| $, are linearly independent, $\widehat{\gamma}$ is the unique solution of the system  $\lbrace a_i^T\rho = b_i : i =1,2,\dots,|E| \rbrace $. Therefore,  $\widehat{\gamma}$ is the unique optimal solution  for $\Mod_{1,\sigma}(\Gamma)$.
	
	Finally, we show that this choice of $\sigma$ belongs to  $\mathbb{R}^E_{> 0}$. 
	Notice that $ a_i \in \R^E_{\geq0}$, for $i =1,2,\dots,|E| $, therefore $\si  \in \R^E_{\geq0}$. Moreover, since the determinant of the $|E| \times |E|$ matrix $A$ formed by $|E|$ columns $ a_i$, for $ i =1,2,\dots,|E| $, is nonzero (by linearly independence), there is no row in $A$ containing all zeros. Thus, $\si >0$, this is because $\si$ is the vector of row sums for $A$. Therefore, $\si  \in \R^E_{>0}$.
\end{proof}

\section{Modulus of feasible partitions}\label{sec:beu}
\subsection{Beurling partitions}
Let $G=(V,E)$ be a connected graph. Let $\Ga = \Ga_G$ be the spanning tree family of $G$. Let $\Phi$ be the family of all feasible partitions $P$ of $G$, with usage vectors defined as in (\ref{usage2}). Let $\eta^*$ be the optimal density for $\Mod_{2}(\Phi)$.
For any feasible partition $P$ with the cut set $E_P$, denote  $\eta^*(E_P):= \sum\limits_{e \in E_P} \eta^*(e)$. In this section, we introduce the notion of {\it Beurling partitions}.

\begin{definition}\label{def:beurling-partition}
	Let $G=(V,E)$ be a connected graph. Let $\Phi$ be the family of all feasible partitions $P$ of $G$, with usage matrix $\widetilde{\cN}$ defined as in (\ref{usage2}). Let $\eta^*$ be the optimal density for $\Mod_{2}(\Phi)$. A  feasible partition $P$ of $G$ is said to be a {\it Beurling partition} if it has minimal usage with respect to $\eta^*$, equivalently, $ \widetilde{\cN}(P,\cdot)\eta^*=1$. In other words, $P$ is a Beurling partition if and only if  $\eta^*(E_P)= k_P -1$.
\end{definition}
\begin{remark}
The trivial partition of the node set $V$ is a Beurling  partition because $\eta^*(E)=|V|-1$.
\end{remark}
Beurling partitions are closely related to fair trees, the relation between Beurling partitions and  fair trees is described in the following theorem.
\begin{theorem} \label{fair} Let $G=(V,E)$ be a connected graph. Let $P$ be a feasible partition of $G$. Then, $P$ is a Beurling partition if and only if the restriction of every fair tree $\gamma$ onto $E_P$ is a spanning tree of $G_P$.
\end{theorem}
\begin{proof}
We recall that for every spanning tree $\ga$ of $G$, the restriction of $\ga$ onto $E_P$ forms a connected spanning subgraph of the shrunk graph  $G_P$. Let $\eta^*$ be the optimal density for $\Mod_{2}(\Phi)$ and let $\mu^{\ast}$ be any optimal pmf  for the $\MEO(\Ga)$ problem. Let $\cN$ be the usage matrix associated with $\Ga$ where the usage vector $\cN(\ga,\cdot)^T$ of each spanning tree $\ga$ is the indicator function $\ones_{\ga}$. Also, let $\widetilde{\cN}$ be the usage matrix for $\Phi$ as defined in (\ref{usage2}).
We have that $\widetilde{\cN}(P,\cdot)$ is a $1 \times |E|$ matrix, $\cN^T$ is a $|E| \times |\Ga|$ matrix, $\mu^{\ast}$ is a $|\Ga| \times 1$ matrix and $\eta^*$ is a $|E| \times 1$ matrix. By Theorem \ref{meomod}, we have $\eta^{\ast} = \cN^T\mu^{\ast}$. Then, we have the following matrix multiplication,
\begin{align*}
\widetilde{\cN}(P,\cdot)\eta^* &=   \widetilde{\cN}(P,\cdot)\left( \cN^T\mu^{\ast} \right) \\
         & = \left(\widetilde{\cN}(P,\cdot) \cN^T\right)\mu^{\ast}\\
         &= \sum\limits_{\ga \in \Ga} \left[ \left( \frac{1}{k_P-1}\ones_{E_P}\cdot  			\ones_{\ga} \right)\mu^*(\ga) \right],      
\end{align*} 
where $a\cdot b$ is the dot product between any vectors $a$ and $b$. Since $\ones_{E_P}\cdot\ones_{\ga} = |E_P \cap \ga| \geq k_P-1$  for any feasible partition $P$ and any spanning tree $\ga$, we obtain that $\widetilde{\cN}(P,\cdot)\eta^*=1$ if and only if $|E_P \cap \ga|= k_P-1$ for any spanning tree $\ga$ satisfying $\mu^*(\ga)>0$. In other words, $\widetilde{\cN}(P,\cdot)\eta^*=1$ if and only if the restriction of every fair tree $\gamma$ onto $E_P$ is a spanning tree of $G_P$.
\end{proof}

\subsection{Serial rule}
Let $G=(V,E)$ be a connected graph. Let $\Ga = \Ga_G$ be the spanning tree family of $G$. 
Given a feasible partition $P$, let $\Gamma^P$ be the set of all spanning trees of $G$ whose restriction onto  $E_P$ is a spanning tree of the shrunk graph $G_P$. Let $\Ga^f_G$ be the set of all fair trees of $G$. By Theorem \ref{fair}, $P$ is  a Beurling partition if and only if $\Ga^f_G \subset \Gamma^P$. Next, we study the relation between $\Ga^P$ and $P$, for arbitrary feasible partitions.
\begin{theorem}\label{par}
Let $G=(V,E)$ be a connected graph. Let $P= \left\{ V_1,V_2,...,V_{k_P} \right\}$ be a feasible partition of $G$.  Let $G_i = (V_i,E_i)$ be the subgraph induced by $V_i$ for each $i=1,2,\dots,k_P$. Let $\Gamma^P$ be the set of all spanning trees of $G$ whose restriction onto  $E_P$ is a spanning tree of the shrunk graph $G_P$. Then, we have
\bi
\item[(i)] The restrictions map trees to trees, meaning that
\begin{equation}\label{eq:part1}
  \psi_{E_P}(\Gamma^P) = \Gamma_{G_P} ,  \quad \psi_{E_i}(\Gamma^P) = \Gamma_{G_i}    \quad \forall i=1,2,\dots,k_P.
\end{equation} 
\item[(ii)] The partition $ \lbrace E_1,E_2,...,E_{k_{P}},E_P \rbrace $ of the edge set $E$  divides  $\Gamma^P$, namely,
\begin{align}  \label{tp0}
\Gamma^P &= \psi_{E_1}(\Gamma^P) \oplus \psi_{E_2}(\Gamma^P) \oplus ... \oplus \psi_{E_{k_{P}}}(\Gamma^P) \oplus \psi_{E_P} (\Gamma^P) \notag \\
&= \Ga_{G_1} \oplus \Ga_{G_2} \oplus ... \oplus  \Ga_{G_{k_P}} \oplus  \Ga_{G_P}.
\end{align}
\ei

\end{theorem}
\begin{proof}
We start by proving (i). Let $\Ga_{G_P}$ be the set of all spanning trees of $G_P$. 
By the definition of $\Ga^P$, we have 
\begin{equation}\label{eq:gamp3}
\psi_{E_P}(\Gamma^P) \subset \Gamma_{G_P}.
\end{equation}
To show the other direction, let $\gamma_P$ be an arbitrary spanning tree of $G_P$. For each  $i=1,2,\dots,k_P$, let $\gamma_i$ be an arbitrary spanning tree of $G_i$. 
Consider the subgraph 
\begin{equation}\label{eq:construct-tree}
\gamma := \left(\bigcup\limits_{i =1}^{k_{P}}\gamma_i\right) \cup \gamma_P
\end{equation}
of $G$. 
We have that $\gamma$ is connected, spans $G$ since $\ga_i$ spans $V_i$. Moreover, it has exactly $|V|-1$ edges  because \[|\ga| = (|V_1|-1)+(|V_2|-1)	+... +(|V_{k_{P}}|-1)+(k_{P}-1)=|V|-1.\]
Therefore, $\ga$ is a spanning tree of $G$ and $\psi_{E_P}(\ga) = \ga_P$. That means $\ga \in \Ga^P,$ and we have shown that $\Gamma_{G_P}\subset \psi_{E_P}(\Gamma^P)$. Combining with (\ref{eq:gamp3}) we have proved the first claim in (i).
 
For the second claim in (i), we first show that $ \Gamma_{G_i}\subset\psi_{E_i}(\Gamma^P)$.
The argument is similar as above. Given a tree $\ga_i\in  \Gamma_{G_i}$, we pick trees arbitrarily in $ \Gamma_{G_j}$ for $j\ne i$ and in $\Gamma_{G_P}$. Then, as before, the subgraph $\ga$ constructed as in (\ref{eq:construct-tree}) must be a spanning tree of $G$ and is in $\Ga^P$ by construction. So, $\ga_i\in \psi_{E_i}(\Gamma^P)$.

Conversely, let $\gamma$ be an arbitrary spanning tree in $\Gamma^P$, we have $|\psi_{E_P}(\gamma)| = k_P-1$ and $\psi_{E_i}(\gamma)$ is a spanning forest of $G_i$ for all  $i=1,2,\dots,k_P$. Then,
\begin{align*}
|V|-1= |\gamma| &= |\psi_{E_1}(\gamma)| + | \psi_{E_2}(\gamma)| + ... + |\psi_{E_{k_{P}}}(\gamma)| +| \psi_{E_P}(\gamma)|\\
& \leq (|V_1|-1)+(|V_2|-1)	+... +(|V_{k_{P}}|-1)+(k_{P}-1)\\
&=|V|-1.
\end{align*} 
Hence, $|\psi_{E_i}(\gamma)|= |V_i|-1$ for all  $i=1,2,\dots,k_P$. That implies $\psi_{E_i}(\gamma)$ is a spanning tree of $G_i$. In other words, 
\begin{equation}\label{eq:gamp2}
\psi_{E_i}(\Gamma^P) \subset \Gamma_{G_i} \quad \text{ for all } i=1,2,\dots,k_P.
\end{equation}
This concludes the proof of (i).

Next, we show (ii). The second equality follow by part (i). To show the first equality,
since $ \lbrace E_1,E_2,...,E_{k_{P}},E_P \rbrace $ is a partition of the edge set $E$, we have that for every spanning tree $\ga$ in $\Gamma^P$, it can be decomposed as  $\ga = \psi_{E_1}(\ga)  \cup \psi_{E_2}(\ga)  \cup \dots  \cup \psi_{E_{k_{P}}}(\ga) \cup \psi_{E_P} (\ga).$ So,
\begin{equation}\label{tp1}
\Gamma^P \subset \psi_{E_1}(\Gamma^P) \oplus \psi_{E_2}(\Gamma^P) \oplus ... \oplus \psi_{E_{k_{P}}}(\Gamma^P) \oplus \psi_{E_P} (\Gamma^P).
\end{equation}
Conversely,
 if $\ga_i\in\Ga_{G_i}$ and $\ga_P\in\Ga_{G_P}$, constructing $\ga$ as in (\ref{eq:construct-tree}), we see that
 \begin{equation}\label{tp2}
 \Ga_{G_1} \oplus \Ga_{G_2} \oplus ... \oplus  \Ga_{G_{k_P}} \oplus  \Ga_{G_P} \subset \Gamma^P.
 \end{equation}
 This proves (ii).
\end{proof}
 Recall the following definition from  \cite{pietrofairest}.
\begin{definition}\label{def:restriction-prop}
Given a subgraph $H$ of $G$, we say that $H$ has the {\it restriction property}, if every fair tree $\ga \in \Ga^f_G$ restricts to a spanning tree of $H$. 
 \end{definition}
 Then, we have the following corollary.
\begin{corollary}\label{coro-res}
Let $G=(V,E)$ be a connected graph.  Let $P= \left\{ V_1,V_2,...,V_{k_P} \right\}$ be a feasible partition of $G$.  Let $G_i = (V_i,E_i)$ be the subgraph induced by $V_i$ for $i=1,2,\dots,k_P$. Then, $P$ is a Beurling partition if and only if $G_i$ has the restriction property for  $i=1,2,\dots,k_P$.
\end{corollary}
\begin{proof}

Let $\Gamma^P$ be the set of all spanning trees of $G$ whose restriction onto  $E_P$ is a spanning tree of the shrunk graph $G_P$.  By Theorem \ref{par}, we have  $\psi_{E_i}(\Gamma^P) = \Gamma_{G_i}$  for all  $i=1,2,\dots,k_P$ and   $  \psi_{E_P}(\Gamma^P)= \Gamma_{G_P}. $

Assume that  $P$ is a Beurling  partition. By Theorem \ref{fair}, we have $\Ga^f_G \subset \Ga^P$. Therefore, \[\psi_{E_i}(\Gamma^f_G) \subset \psi_{E_i}(\Gamma^P)= \Gamma_{G_i} \quad \text{ for all } i=1,2,\dots,k_P.\]  In other words,  $G_i$ has the restriction property for  $i=1,2,\dots,k_P$.

Conversely, assume that $G_i$ has the restriction property for  $i=1,2,\dots,k_P$. Then \[\psi_{E_i}(\Gamma^f_G) \subset  \Gamma_{G_i} = \psi_{E_i}(\Gamma^P) \quad \text{ for all } i=1,2,\dots,k_P.\]
 Let $\gamma$ be a fair tree, we have $|\psi_{E_i}(\gamma)| = |V_i|-1$. Then
\begin{align*}
|V|-1= |\gamma| &= |\psi_{E_1}(\gamma)| + | \psi_{E_2}(\gamma)| + ... + |\psi_{E_{k_{P}}}(\gamma)| +| \psi_{E_P}(\gamma)|\\
& = (|V_1|-1|)+(|V_2|-1|)	+... +(|V_{k_{P}}|-1|)+ |\psi_{E_P}(\gamma)|\\
& = |V| -k_P + |\psi_{E_P}(\gamma)|.
\end{align*}
Therefore, $ |\psi_{E_P}(\gamma)| = k_{P}-1$. Hence, $\Ga^f_G \subset \Ga^P$. In other words, $P$ is a Beurling  partition.

\end{proof}
Given $A \subset E$, let $\psi_A$ be the {\it restriction operator}, 
\begin{align*}
	\psi_A: 2^E  & \rightarrow 2^A\\
	\ga \subseteq E & \mapsto \ga \cap A.
\end{align*}
Then, for each  $A \subset E$, $\psi_A$ induces a family of objects $ \psi_{A}(\Ga) = \lbrace \ga\cap A: \ga \in \Ga \rbrace.$

\begin{definition}\label{def:divides-gamma}
	Let $\left\{ E_1,E_2,\dots,E_k \right\} $ be a partition of the edge set $E$. For each $i = 1, \dots,k$, we define an induced family of objects $\Ga_i :=  \psi_{E_i}(\Ga)$.
	We say that a partition $\left\{ E_1,E_2,\dots,E_k \right\} $ of the edge set $E$ {\it divides $\Ga$}, if 
	$\Ga$ coincides with the concatenation
	\begin{equation*}
		\Ga_1 \oplus \Ga_2 \oplus \dots \oplus \Ga_k := \left\{ \bigcup^{k}_{i=1}\ga_i : \ga_i \in  \Ga_i, i =1,2,\dots,k \right\}.
	\end{equation*}
\end{definition}
Given a partition $\left\{ E_1,E_2,\dots,E_k \right\}$ that divides $\Ga$ and a pmf $\mu \in \cP(\Ga)$. For each $i = 1, \dots,k$, define the {\it marginal} $\mu_i \in \cP(\Ga_i)$ as follows,
\[\mu_{i}(\zeta) :=  \sum \left\{   \mu(\zeta) : \ga \in \Ga, \psi_i(\ga)= \zeta \right\} \quad \forall \zeta \in \Ga_i.\]
On the other hand, given measures $\nu_i \in \cP ( \Ga_i)$ for $i =1,2,\dots,k$, define their {\it product measure} in $\cP(\Ga)$ as follows, \[ \left( \nu_1  \oplus \nu_2 \oplus \dots \oplus \nu_k \right) (\ga)  := \nu_1(\zeta_1) \nu_2(\zeta_2) \dots \nu_k(\zeta_k), \]
for all $ \displaystyle \ga = \bigcup^{k}_{i=1}\zeta_i$ where $ \zeta_i \in \Ga_i$, $i =1,2,\dots,k.$

The following theorem gives ways to decompose the graph $G$ and split the $\MEO$ problem of $G$ into $\MEO$ problems for smaller graphs using Beurling partitions.
\begin{theorem}\label{serialmod}
	Let $G=(V,E)$ be a connected graph. Let $\Ga_G$ be the spanning tree family of $G$. Let $\Phi$ be the family of all feasible partitions $P$ of $G$, with usage vectors defined as in (\ref{usage2}).  Let $\eta^*$ be the optimal density for $\Mod_{2}(\Phi)$. Let $P= \left\{ V_1,V_2,...,V_{k_P} \right\}$ be a Beurling  partition of $G$.  Let $G_i = (V_i,E_i)$ be the subgraph induced by $V_i$ for $i=1,2,\dots,k_P$ and let $G_P$ be the shrunk graph of the partition $P$. Then, the minimum expected overlap problem (\ref{meo}) splits as follows:
	\bi
	\item[ (i)] \begin{equation}
		\MEO (\Ga_G) = \MEO (\Ga_{G_1}) + \MEO (\Ga_{G_2}) +\dots + \MEO (\Ga_{G_{k_P}})+ \MEO (\Ga_{G_P});
	\end{equation}
	\item[ (ii)] A pmf $\mu \in \cP(\Ga_G)$  is optimal for $\MEO(\Ga_G)$ if and only if its marginal pmfs $\mu_i \in \cP(\Ga_{G_i})$, $i=1,2,\dots,k_P$ are optimal for 
	$\MEO(\Ga_{G_i})$ respectively and its marginal pmf $\mu_P \in \cP(\Ga_{G_P})$ is  optimal for $\MEO(\Ga_{G_P})$ (where we have identified $E_P$ with $E(G_P)$);
	\item[ (iii)] Conversely, given  pmfs $\nu_i \in \cP(\Ga_{G_i})$ that are optimal for $\MEO(\Ga_{G_i})$  for $i=1,2,\dots,k_P$ and given a pmf $\nu_P$ that is optimal for  $\MEO(\Ga_{G_P})$, then $\nu_1 \oplus \nu_2 \oplus + \dots \oplus \nu_{k_P} \oplus \nu_P $ is an optimal pmf in $\cP(\Ga_G) $ for $\MEO(\Ga_G)$; 
	\item[ (iv)] The restriction of $\eta^*$ onto $E_i$ is optimal for $\Mod_2(\widehat{\Ga_{G_i}}) $ for $i = 1,2,\dots,k_P$ and the  restriction of $\eta^*$ onto $E_P$  is optimal for $\Mod_2(\widehat{\Ga_{G_P}})$. 
	\ei
\end{theorem}

\begin{proof}[Proof for Theorem \ref{serialmod}]
Let $\Gamma^P$ be the set of all spanning trees of $G$ whose restriction onto  $E_P$ is a spanning tree of the shrunk graph $G_P$.  By Theorem \ref{par}, we have

\[\Ga^P = \Ga_{G_1} \oplus \Ga_{G_2} \oplus \dots \Ga_{G_{k_P}} \oplus \Ga_{G_P}.\]
Since $P$ is a Beurling partition, by Theorem \ref{fair}, we have $\Ga^f_G \subset \Ga^P$. Hence, any optimal pmf $\mu^* \in \cP(\Ga_G)$ for $\MEO(\Ga_G)$ must necessarily lie in $\cP(\Ga^P)$. Therefore, $\MEO(\Ga_G) = \MEO(\Ga^P)$. The rest of the proof follows by \cite[Theorem 3.3]{pietrofairest}.
\end{proof}

\begin{corollary} Let $G=(V,E)$ be a connected graph. Let $\Phi$ be the family of all feasible partitions $P$ of $G$, with usage vectors defined as in (\ref{usage2}).  Let $\eta^*$ be the optimal density for $\Mod_{2}(\Phi)$. Given a connected vertex-induced  subgraph $H = (V_H,E_H)$ of $G$, if $\eta^*(E_H)= |V_H|-1$, then $\eta^*_{\vert_{E_H}}$ is the optimal density for $\Mod_2(\widehat{\Ga_H})$. 
\end{corollary}

\begin{proof}
Consider the partition $P_H := \left\{ V_H \right\} \cup  \left\{  \left\{ v \right\}: v \notin V_H  \right\}$. Then, $P_H$ is feasible since $H$ is connected. Moreover, $k_{P_H}=|V\setminus V_H|+1$ and  $E_{P_H}=E\setminus E_H$. Assume that $\eta^*(E_H)= |V_H|-1$, we have
\[\eta^*(E \setminus E_H)= \eta^*(E) - \eta^*(E_H) = (|V|-1)- (|V_H|-1) = |V|- |V_H| = k_{P_H}-1. \]
Hence, by Definition \ref{def:beurling-partition}, $P_H$ is a Beurling partition. By Theorem \ref{serialmod} (iv), $\eta^*_{\vert_{E_H}}$ is the optimal density for $\Mod_2(\widehat{\Ga_H})$. 
\end{proof}

Theorem \ref{serialmod} gives us  a way to decompose the graph $G$ into a set of vertex-induced subgraphs and a shrunk graph using Beurling partitions. 

Let $P= \left\{ V_1,V_2,...,V_{k_P} \right\}$ be a Beurling  partition of $G$. Let $G_i = (V_i,E_i)$ be the subgraph induced by $V_i$ for $i=1,2,\dots,k_P$. Then, the graph $G$ can be decomposed into $G_i$ for $i=1,2,\dots,k_P$ and the shrunk graph $G_P$, in the sense that  we decompose the edge set $E$ into the partition $\lbrace E_1,E_2,...,E_{k_{P}},E_P \rbrace $ and preserve the optimal edge usage $\eta^*(e)$ by Theorem \ref{serialmod}.

\section{Combinatorial properties and Beurling partitions}

Among all Beurling partitions, there are two that are related to combinatorial notions. The first one represents the maximum denseness of the graph $G$ and the second one represents the strength of the graph $G$.

\subsection{The maximum denseness of a graph}\label{sec4}

First, we recall the concept of homogeneous core introduced in \cite{pietrofairest}.
\begin{definition}\label{def:homogeneous-core}
A connected subgraph $H$ of $G$  is called a {\it homogeneous core} of $G$, if it satisfies the following properties:
\bi
\item[ (i)] $H$ has at least one edge.
\item[ (ii)]$H$ has the restriction property, see Definition \ref{def:restriction-prop}.
\item[ (iii)]$H$ is a vertex-induced subgraph of $G$.
\item[ (iv)] $H$ is itself a homogeneous graph, see Definition \ref{def:homogeneous}.
\ei
\end{definition}
It was shown, in \cite[Theorem 5.8]{pietrofairest}, that any maximum denseness subgraph of $G$ is a homogeneous core. Let $\eta^*$  be the optimal density for $\Mod_2(\Phi)$. Let 
\begin{equation}\label{eq:e-min}
E_{min} :=\left\{ e\in E: \eta^*(e)=\min\limits_{e' \in E}\eta^*(e')=:\eta^*_{min}\right\}.
\end{equation}
Let $H_{ min}$ be the subgraph induced by $E_{ min}$. Then, any connected component $H$ of $H_{ min}$ is a homogeneous core  \cite[Theorem 5.2]{pietrofairest} and is a maximum denseness subgraph of $G$  \cite[Corollary 5.10]{pietrofairest}. Moreover,

\begin{equation}\label{etamin2}
\eta^*_{min} =\frac{1}{\theta(H)}= \frac{1}{D(G)}.
\end{equation}

We reinterpret these results in terms of Beurling partitions in the following theorem.

\begin{theorem}\label{etamin}
Let $G=(V,E)$ be a connected multigraph containing at least one edge. Let $\eta^*$  be the optimal density for $\Mod_2(\Phi)$. Let $E_{ min}$ be defined as in (\ref{eq:e-min}). Then,  there exists a Beurling partition $P_{ min}$ such that $E_{P_{ min}} = E \setminus E_{ min}$.
\end{theorem}

\begin{proof}
Let $H_{ min} = (V_{H_{ min}},E_{ min})$ be the subgraph induced by $E_{ min}$. Assume that $H_{ min}$ has $h$ connected components $H_i =  (V_{H_i},E_{H_i}) $, $i = 1,2,\dots,h.$ 
  By \cite[ Theorem 5.2]{pietrofairest}, every $H_i$ is a homogeneous core.
  
 Consider the partition $ P_{min} :=\left\{ V_{H_1}, V_{H_2},\dots  V_{H_h} \right\} \cup \left\{ \left\{ v\right\} : v \notin  V_{H_{ min}} \right\}$ of the node set $V$. The partition $ P_{ min}$ is feasible since every $H_i$ is connected. $E_{P_{ min}} = E \setminus E_{ min}$ since every $H_i$ is a vertex-induced. Finally, since every $H_i$ has the restriction property, by Corollary   \ref{coro-res}, $P_{ min}$ is a Beurling partition.
\end{proof}

\subsection{The strength of a graph}\label{sec5}
The second important Beurling partition has been described in  \cite[Theorem 4.10]{pietrofairest}. We restate that theorem in terms of Beurling partitions as follows.
\begin{theorem}\label{etamax}
Let $G=(V,E)$ be a connected graph. Let $\eta^*$  be the optimal density for $\Mod_2(\Phi)$. Let 
\begin{equation}\label{eq:e-max}
E_{ max} :=\left\{ e\in E: \eta^*(e)=\max\limits_{e' \in E}\eta^*(e')=:\eta^*_{max}\right\}.
\end{equation}  Then, there exists a Beurling partition $P_{ max}$ such that $E_{P_{ max}} = E_{ max}$. Moreover, $P_{ max}$ is a critical partition for the strength problem, see Definition \ref{def:strength-pb}, and 
\begin{equation}\label{etamax2}
\eta^*_{max} = \frac{1}{w(P_{ max})} = \frac{1}{S(G)},
\end{equation}
where $w(P_{max})$ is the weight defined in (\ref{eq:weight-partition}).
\end{theorem}
\begin{corollary}\label{shrunkhom}
Let $G=(V,E)$ be a connected graph. Let $P_{ max}$ be the Beurling partition constructed in Theorem \ref{etamax}. Then, the shrunk graph $G_{P_{ max}}$ is homogeneous.
\end{corollary}
\begin{proof}
  We have that $\eta^*$ is constant in $E_{P_{ max}}$. By Theorem \ref{serialmod}, the  restriction of $\eta^*$ onto $E_{P_{ max}}$  is optimal for $\Mod_2(\widehat{\Ga}_{G_{P_{ max}}})$. Therefore, the shrunk graph $G_{P_{ max}}$ is homogeneous.
\end{proof}

Let $P$ and $P'$ be two partitions of the node set $V$. The partition $P$ is said to be {\it finer} than  $P'$, if every element in $P'$ is a union of elements in $P$. The trivial partition is the finest partition among all partitions. We will show that $P_{max}$ is the finest critical partition for the strength problem. 

\begin{lemma}\label{finer}
Let $G= (V,E)$ be a connected graph. Let $P$ and $P'$ be two feasible partitions of the node set $V$ with the cut sets $E_P$ and $E_{P'}$, respectively. Then  $P$ is finer than  $P'$ if and only if $E_{P'} \subset E_{P}.$
\end{lemma}
\begin{proof}
Suppose $P= \left\{ V_1,V_2,...,V_{k_P} \right\}$ and  $P'= \left\{ V'_1,V'_2,...,V'_{k_{P'}} \right\}$. Assume that $P$ is finer than  $P'$, then each $V'_i$ is a union of elements in $P$ for $i =1,2,\dots,k_P.$ Let $e \in E_{P'}$, then $e$ connects two vertices which belong to two different $V'_i$. Since each $V'_i$ is a union of elements in $P$, the endpoints of $e$ belong to two different $V_i$. Therefore, $E_{P'} \subset E_{P}.$

Assume that $E_{P'} \subset E_{P}.$  Let $G'_i = (V'_i,E'_i)$ be the subgraph induced by $V'_i$ for $i=1,2,\dots,k_{P'}$ and let $G_i = (V_i,E_i)$ be the subgraph induced by $V_i$ for $i=1,2,\dots,k_P$. First, we remove all the edges in $E_{P'}$ from $G$, then we are left with the graph $G\setminus E_{P'}$, the graph $G\setminus E_{P'}$ is the disjoint union of all the $G'_i$ and each $G'_i$ is a connected component of $G\setminus E_{P'}$.
 Next, we remove all the edges in $E_P \setminus E_{P'}$ from $G\setminus E_{P'}$. Then, we are left with the graph $G\setminus E_{P}$. 
 After removing the edges in  $E_P \setminus E_{P'}$, each $G'_i$ breaks into a disjoint union of connected components of $G\setminus E_{P}$. Notice that  $G\setminus E_{P}$ is  the disjoint union of all $G_i$ and each $G_i$ is a connected component of $G\setminus E_{P}$. Therefore, each $V'_i$ is a union of elements in $P$. In other words,  $P$ is finer than  $P'$.

\end{proof}
\begin{theorem}\label{finest} Let $G=(V,E)$ be a connected graph. Let $P_{ max}$ be the Beurling  partition constructed in Theorem \ref{etamax}. If $P$ is a critical partition of $G$, then $P_{ max}$ is finer than $P$. We say that $P_{ max}$ is the {\it finest critical partition}.
\end{theorem}
\begin{proof}
By  \cite[Theorem 2.5]{polynomial}, we have $E_P \subset E_{P_{ max}}$. Therefore, by Lemma \ref{finer}, $P_{ max}$ is finer than $P$.
\end{proof}
\subsection{A characterization of  homogeneous graphs}\label{sec6}
Let  $G= (V,E)$ be a connected graph. 
Since $G$ is its own vertex-induced connected subgraph, we have $\theta(G)\leq D(G)$. Since $\theta(G)$ is the weight of the trivial partition, we have $S(G)\leq \theta(G)$.
By (\ref{etamin2}) and (\ref{etamax2}), the following holds:
\begin{equation}\label{maxmin}
 \frac{1}{\eta^*_{ max}} = S(G)\leq \theta(G)\leq D(G) = \frac{1}{\eta^*_{ min}}.
 \end{equation}

The next result gives several equivalent characterizations of homogeneous graphs, see Definition \ref{def:homogeneous}.
\begin{theorem}\label{homo} 
	Let $G=(V,E)$ be a connected graph and let $\Ga = \Ga_G$ be the spanning tree family of $G$. Let $S(G)$ be the strength of $G$, let $D(G)$ be the maximum denseness of $G$, and let $\theta(G)$ be the denseness of $G$. Then, the following statements are equivalent:
	
	\bi 
	\item[ (i)]  $G$ is homogeneous, in the sense of Definition \ref{def:homogeneous}.
	\item[ (ii)]   $S(G)=\theta(G).$
	\item[ (iii)]   $D(G) =\theta(G)$.
	\item[ (iv)]   $S(G)=D(G)$.
	\ei
\end{theorem}

\begin{proof}[Proof for Theorem \ref{homo}]
 We recall that $G$ is homogeneous if and only if  $\eta^*_{min}=\eta^*_{max}$. Thus, by (\ref{maxmin}), (i) implies (ii), (iii), and (iv), and also (iv) implies (i), (ii), and (iii).
Therefore, it is enough to show that (ii) implies (i) and (iii) implies (i).
 
 Assume that $S(G)=\theta(G)$. Then, the trivial partition is a critical partition. Let $P_{ max}$ be the Beurling partition constructed in Theorem \ref{etamax}. By Theorem \ref{finest}, $P_{ max}$ is the trivial partition and $E = E_{P_{ max}}$. Therefore, $\eta^*$ is constant.

Assume that $D(G) =\theta(G)$. Then, $G$ is a maximum denseness subgraph of $G$ which implies $G$ is a homogeneous core. Therefore, $G$ is homogeneous.
\end{proof}
\subsection{Deflation processes}\label{sec:deflation}
We recall the deflation process by shrinking homogeneous cores in \cite{pietrofairest}. Let $H_{ min}$ be the subgraph induced by $E_{ min}$, where $E_{ min}$ is defined as in  (\ref{eq:e-min}). Then, any connected component $H$ of $H_{ min}$ is a homogeneous core \cite[Theorem 5.2]{pietrofairest}, see Definition \ref{def:homogeneous-core}.

In \cite{pietrofairest}, it was shown that when we start from a connected graph $G$, there exists a homogeneous core $H \subset G$ and $H =(V_H,E_H)$. This divides the $\MEO$ problem on $G$ into two subproblems: an $\MEO$ problem on $H$ and one on $G/H$ where $G/H$ is the shrunk graph of the partition  $P_H$ which is formed by $V_H$ and all singletons $v$ where $ v \notin V_H$. It can be shown that $P_H$ is a Beurling partition using Theorem \ref{fair} and Definition \ref{def:homogeneous-core}.

We have

\[ \MEO(\Ga_{H}) = \frac{(|V_H|-1)^2}{|E_{H}|},\]
and

\[ \MEO (\Ga_G) = \MEO(\Ga_{G/H}) + \frac{(|V_H|-1)^2}{|E_{H}|}.\]

Iterating this procedure on $\Ga_{G/H}$ gives a decomposition of $G$  into a sequence of shrunk graphs, eventually terminating when a final homogeneous graph is shrunk to a single vertex. By tracing back  through the deflation process, one can construct an optimal pmf for the $\MEO$ problem on $G$ from corresponding optimal pmfs on the various homogeneous cores shrunk during the deflation process.
This deflation process in \cite{pietrofairest} is described by shrinking each homogeneous core which is a connected component of $H_{min}$ in each step. 

In this paper, we describe the deflation process by decomposing a graph recursively using Beurling partitions such as $P_{max}$ and $P_{min}$. In the deflation process using $P_{min}$, we shrink all the connected components of $H_{min}$ at the same time in each step, thus this one is slightly different from the deflation process given in \cite{pietrofairest}. 

Now, we provide two examples of deflation processes using Beurling partitions.

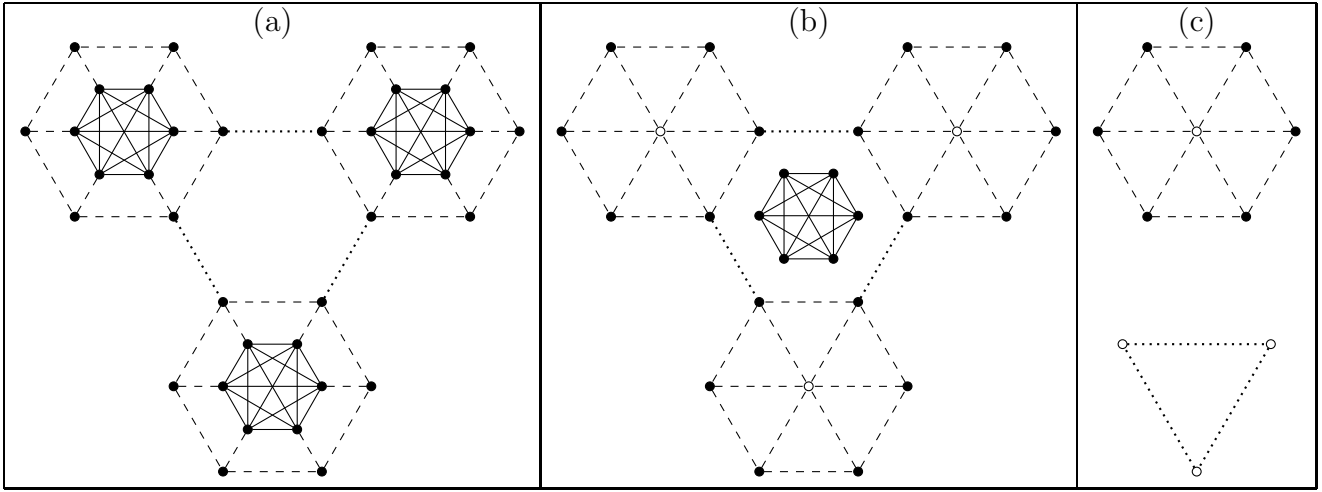
\begin{figure}[t]
	\centering
	\begin{center}
		\begin{tabular}{ |c|c|c| } 
			\hline
			(a) &  (b) & (c) \\
			\begin{tikzpicture}[scale=1.3, auto, node distance=3cm,  thin]
				\begin{scope}[every node/.style={circle,draw=black,fill=black!100!,font=\sffamily\Large\bfseries}]
					
					\node (A) [scale=0.3]at (0,0) {};
					\node (B)[scale=0.3] at (1,0) {};
					\node (C)[scale=0.3] at (1.5,0.87) {};
					\node (D) [scale=0.3]at (1,1.73) {};
					\node (E)[scale=0.3] at (0,1.73) {};
					\node (F)[scale=0.3] at (-0.5,0.87) {};
					\node (G)[scale=0.3] at (0.25,0.43) {};
					\node (H) [scale=0.3]at (0.75,0.43) {};
					\node (I)[scale=0.3] at (1,0.87) {};
					\node (J)[scale=0.3] at (0.75,1.3) {};
					\node (K)[scale=0.3] at (0.25,1.3) {};
					\node (L)[scale=0.3] at (0,0.87) {};
					
					\node (A1) [scale=0.3]at (0+3,0) {};
					\node (B1)[scale=0.3] at (1+3,0) {};
					\node (C1)[scale=0.3] at (1.5+3,0.87) {};
					\node (D1) [scale=0.3]at (1+3,1.73) {};
					\node (E1)[scale=0.3] at (0+3,1.73) {};
					\node (F1)[scale=0.3] at (-0.5+3,0.87) {};
					\node (G1)[scale=0.3] at (0.25+3,0.43) {};
					\node (H1) [scale=0.3]at (0.75+3,0.43) {};
					\node (I1)[scale=0.3] at (1+3,0.87) {};
					\node (J1)[scale=0.3] at (0.75+3,1.3) {};
					\node (K1)[scale=0.3] at (0.25+3,1.3) {};
					\node (L1)[scale=0.3] at (0+3,0.87) {};
					
					\node (A2) [scale=0.3]at (0+1.5,0-2.6) {};
					\node (B2)[scale=0.3] at (1+1.5,0-2.6) {};
					\node (C2)[scale=0.3] at (1.5+1.5,0.87-2.6) {};
					\node (D2) [scale=0.3]at (1+1.5,1.73-2.6) {};
					\node (E2)[scale=0.3] at (0+1.5,1.73-2.6) {};
					\node (F2)[scale=0.3] at (-0.5+1.5,0.87-2.6) {};
					\node (G2)[scale=0.3] at (0.25+1.5,0.43-2.6) {};
					\node (H2) [scale=0.3]at (0.75+1.5,0.43-2.6) {};
					\node (I2)[scale=0.3] at (1+1.5,0.87-2.6) {};
					\node (J2)[scale=0.3] at (0.75+1.5,1.3-2.6) {};
					\node (K2)[scale=0.3] at (0.25+1.5,1.3-2.6) {};
					\node (L2)[scale=0.3] at (0+1.5,0.87-2.6) {};
					
				\end{scope}
				\begin{scope}[every edge/.style={draw=black,thin}]
					
					\draw  (G) edge node{} (H); 
					\draw  (H) edge node{} (I); 
					\draw  (I) edge node{} (J);
					\draw  (J) edge node{} (K);
					\draw  (K) edge node{} (L); 
					\draw  (L) edge node{} (G);

					\draw  (G1) edge node{} (H1); 
					\draw  (H1) edge node{} (I1); 
					\draw  (I1) edge node{} (J1);
					\draw  (J1) edge node{} (K1);
					\draw  (K1) edge node{} (L1); 
					\draw  (L1) edge node{} (G1);
					
					\draw  (G2) edge node{} (H2); 
					\draw  (H2) edge node{} (I2); 
					\draw  (I2) edge node{} (J2);
					\draw  (J2) edge node{} (K2);
					\draw  (K2) edge node{} (L2); 
					\draw  (L2) edge node{} (G2); 
					
					\draw  (G) edge node{} (I);
					\draw  (G) edge node{} (J);
					\draw  (G) edge node{} (K);
					\draw  (H) edge node{} (J);
					\draw  (H) edge node{} (K);
					\draw  (H) edge node{} (L);
					\draw  (I) edge node{} (K);
					\draw  (I) edge node{} (L);
					\draw  (J) edge node{} (L);

					\draw  (G1) edge node{} (I1);
					\draw  (G1) edge node{} (J1);
					\draw  (G1) edge node{} (K1);
					\draw  (H1) edge node{} (J1);
					\draw  (H1) edge node{} (K1);
					\draw  (H1) edge node{} (L1);
					\draw  (I1) edge node{} (K1);
					\draw  (I1) edge node{} (L1);
					\draw  (J1) edge node{} (L1);
					
					\draw  (G2) edge node{} (I2);
					\draw  (G2) edge node{} (J2);
					\draw  (G2) edge node{} (K2);
					\draw  (H2) edge node{} (J2);
					\draw  (H2) edge node{} (K2);
					\draw  (H2) edge node{} (L2);
					\draw  (I2) edge node{} (K2);
					\draw  (I2) edge node{} (L2);
					\draw  (J2) edge node{} (L2);
					
				\end{scope}
				
				\begin{scope}[every edge/.style={draw=black,dashed}]
					
					\draw  (A) edge node{} (B);
					\draw  (B) edge node{} (C);
					\draw  (C) edge node{} (D);
					\draw  (D) edge node{} (E);
					\draw  (E) edge node{} (F);
					\draw  (F) edge node{} (A);
					
					\draw  (A1) edge node{} (B1);
					\draw  (B1) edge node{} (C1);
					\draw  (C1) edge node{} (D1);
					\draw  (D1) edge node{} (E1);
					\draw  (E1) edge node{} (F1);
					\draw  (F1) edge node{} (A1);
					
					\draw  (A2) edge node{} (B2);
					\draw  (B2) edge node{} (C2);
					\draw  (C2) edge node{} (D2);
					\draw  (D2) edge node{} (E2);
					\draw  (E2) edge node{} (F2);
					\draw  (F2) edge node{} (A2);
					
					\draw  (A) edge node{} (G);
					\draw  (B) edge node{} (H);
					\draw  (C) edge node{} (I);
					\draw  (D) edge node{} (J);
					\draw  (E) edge node{} (K);
					\draw  (F) edge node{} (L);
					
					\draw  (A1) edge node{} (G1);
					\draw  (B1) edge node{} (H1);
					\draw  (C1) edge node{} (I1);
					\draw  (D1) edge node{} (J1);
					\draw  (E1) edge node{} (K1);
					\draw  (F1) edge node{} (L1);
					
					\draw  (A2) edge node{} (G2);
					\draw  (B2) edge node{} (H2);
					\draw  (C2) edge node{} (I2);
					\draw  (D2) edge node{} (J2);
					\draw  (E2) edge node{} (K2);
					\draw  (F2) edge node{} (L2);
					
				\end{scope}
				
				\begin{scope}[every edge/.style={draw=black,thick, dotted}]
					
					\draw  (C) edge node{} (F1);
					\draw  (B) edge node{} (E2);
					\draw  (D2) edge node{} (A1);
				\end{scope}
			\end{tikzpicture}
			& 
			\begin{tikzpicture}[scale=1.3, auto, node distance=3cm,  thin]
				\begin{scope}[every node/.style={circle,draw=black,fill=black!100!,font=\sffamily\Large\bfseries}]
					
					\node (A) [scale=0.3]at (0,0) {};
					\node (B)[scale=0.3] at (1,0) {};
					\node (C)[scale=0.3] at (1.5,0.87) {};
					\node (D) [scale=0.3]at (1,1.73) {};
					\node (E)[scale=0.3] at (0,1.73) {};
					\node (F)[scale=0.3] at (-0.5,0.87) {};
					
					\node (A1) [scale=0.3]at (0+3,0) {};
					\node (B1)[scale=0.3] at (1+3,0) {};
					\node (C1)[scale=0.3] at (1.5+3,0.87) {};
					\node (D1) [scale=0.3]at (1+3,1.73) {};
					\node (E1)[scale=0.3] at (0+3,1.73) {};
					\node (F1)[scale=0.3] at (-0.5+3,0.87) {};
					
					\node (A2) [scale=0.3]at (0+1.5,0-2.6) {};
					\node (B2)[scale=0.3] at (1+1.5,0-2.6) {};
					\node (C2)[scale=0.3] at (1.5+1.5,0.87-2.6) {};
					\node (D2) [scale=0.3]at (1+1.5,1.73-2.6) {};
					\node (E2)[scale=0.3] at (0+1.5,1.73-2.6) {};
					\node (F2)[scale=0.3] at (-0.5+1.5,0.87-2.6) {};

					\node (G2)[scale=0.3] at (0.25+1.5,0.43-2.6 -0.43 -0.43+2.6) {};
					\node (H2) [scale=0.3]at (0.75+1.5,0.43-2.6-0.43 -0.43+2.6) {};
					\node (I2)[scale=0.3] at (1+1.5,0.87-2.6-0.43 -0.43+2.6) {};
					\node (J2)[scale=0.3] at (0.75+1.5,1.3-2.6-0.43-0.43+2.6) {};
					\node (K2)[scale=0.3] at (0.25+1.5,1.3-2.6-0.43-0.43+2.6) {};
					\node (L2)[scale=0.3] at (0+1.5,0.87-2.6-0.43-0.43+2.6) {};
				\end{scope}
				
				\begin{scope}[every node/.style={circle,draw=black,font=
						\sffamily\Large\bfseries}]
					
					\node (W1) [scale=0.3]at (0.5,0.87) {};
					\node (U1) [scale=0.3]at (3.5,0.87) {};
					\node (V1) [scale=0.3]at (2,-1.73) {};
				\end{scope}

				\begin{scope}[every edge/.style={draw=black,thin}]
					

					
					\draw  (G2) edge node{} (H2); 
					\draw  (H2) edge node{} (I2); 
					\draw  (I2) edge node{} (J2);
					\draw  (J2) edge node{} (K2);
					\draw  (K2) edge node{} (L2); 
					\draw  (L2) edge node{} (G2); 
					

					
					\draw  (G2) edge node{} (I2);
					\draw  (G2) edge node{} (J2);
					\draw  (G2) edge node{} (K2);
					\draw  (H2) edge node{} (J2);
					\draw  (H2) edge node{} (K2);
					\draw  (H2) edge node{} (L2);
					\draw  (I2) edge node{} (K2);
					\draw  (I2) edge node{} (L2);
					\draw  (J2) edge node{} (L2);
					
				\end{scope}
				
				\begin{scope}[every edge/.style={draw=black,dashed}]
					
					\draw  (A) edge node{} (B);
					\draw  (B) edge node{} (C);
					\draw  (C) edge node{} (D);
					\draw  (D) edge node{} (E);
					\draw  (E) edge node{} (F);
					\draw  (F) edge node{} (A);
					
					\draw  (A1) edge node{} (B1);
					\draw  (B1) edge node{} (C1);
					\draw  (C1) edge node{} (D1);
					\draw  (D1) edge node{} (E1);
					\draw  (E1) edge node{} (F1);
					\draw  (F1) edge node{} (A1);
					
					\draw  (A2) edge node{} (B2);
					\draw  (B2) edge node{} (C2);
					\draw  (C2) edge node{} (D2);
					\draw  (D2) edge node{} (E2);
					\draw  (E2) edge node{} (F2);
					\draw  (F2) edge node{} (A2);
					
					\draw  (A) edge node{} (W1);
					\draw  (B) edge node{} (W1);
					\draw  (C) edge node{} (W1);
					\draw  (D) edge node{} (W1);
					\draw  (E) edge node{} (W1);
					\draw  (F) edge node{} (W1);
					
					\draw  (A1) edge node{} (U1);
					\draw  (B1) edge node{} (U1);
					\draw  (C1) edge node{} (U1);
					\draw  (D1) edge node{} (U1);
					\draw  (E1) edge node{} (U1);
					\draw  (F1) edge node{} (U1);
					
					\draw  (A2) edge node{} (V1);
					\draw  (B2) edge node{} (V1);
					\draw  (C2) edge node{} (V1);
					\draw  (D2) edge node{} (V1);
					\draw  (E2) edge node{} (V1);
					\draw  (F2) edge node{} (V1);
					
				\end{scope}
				
				\begin{scope}[every edge/.style={draw=black,thick, dotted}]
					
					\draw  (C) edge node{} (F1);
					\draw  (B) edge node{} (E2);
					\draw  (D2) edge node{} (A1);
				\end{scope}
			\end{tikzpicture} &  
			\begin{tikzpicture}[scale=1.3, auto, node distance=3cm,  thin]
				\begin{scope}[every node/.style={circle,draw=black,font=
						\sffamily\Large\bfseries}]

					\node (U)[scale=0.3] at (0.5,0.87) {};
				\end{scope}
				\begin{scope}[every node/.style={circle,draw=black,fill=black!100!,font=\sffamily\Large\bfseries}]
					\node (A) [scale=0.3]at (0,0) {};
					\node (B)[scale=0.3] at (1,0) {};
					\node (C)[scale=0.3] at (1.5,0.87) {};
					\node (D) [scale=0.3]at (1,1.73) {};
					\node (E)[scale=0.3] at (0,1.73) {};
					\node (F)[scale=0.3] at (-0.5,0.87) {};
					
				\end{scope}

				\begin{scope}[every node/.style={circle,draw=black,font=
						\sffamily\Large\bfseries}]
					
					\node (W1) [scale=0.3]at (-0.25,-1.3) {};
					\node (U1) [scale=0.3]at (1.25,-1.3) {};
					\node (V1) [scale=0.3]at (0.5,-2.6) {};
				\end{scope}

				\begin{scope}[every edge/.style={draw=black,thick, dotted}]
					
					\draw  (W1) edge node{} (U1);
					\draw  (U1) edge node{} (V1);
					\draw  (V1) edge node{} (W1);
				\end{scope}

				\begin{scope}[every edge/.style={draw=black,dashed}]
					\draw  (A) edge node{} (B);
					\draw  (B) edge node{} (C);
					\draw  (C) edge node{} (D);
					\draw  (D) edge node{} (E);
					\draw  (E) edge node{} (F);
					\draw  (F) edge node{} (A);

					\draw  (A) edge node{} (U);
					\draw  (B) edge node{} (U);
					\draw  (C) edge node{} (U);
					\draw  (D) edge node{} (U);
					\draw  (E) edge node{} (U);
					\draw  (F) edge node{} (U);
				\end{scope}
			\end{tikzpicture}
			\\
			\hline
		\end{tabular}
	\end{center}
	\caption{Deflation process using the Beurling partition $P_{ min}$. Edges are styled according to the optimal density $\eta^*$. Solid edges are present with $\eta^* = \frac{1}{3}$, dashed edges with $\eta^* = \frac{1}{2}$ and dotted edges with  $\eta^* = \frac{2}{3}$.}\label{figure2}
\end{figure}

Figure \ref{figure2} describes an example of the deflation process using the Beurling partition $P_{ min}$. Consider a graph $G$ in Figure \ref{figure2}(a). Let $H_{ min} = (V_{H_{ min}}, E_{ min})$ be the subgraph induced by the least-used edges (those solid edges with probability $\frac{1}{3}$). We have $H$ has $3$ connected components $H_i =  (V_{H_i},E_{H_i}) $, $i = 1,2,3$. By \cite[Theorem 5.2]{pietrofairest}, every $H_i$ is a homogeneous core. We shrink each of these homogeneous cores to a single vertex producing a new graph as in Figure \ref{figure2}(b) where each of the cores (the complete graph $K_6$ which is described as in the middle of  \ref{figure2}(b)) is shrunk to a white vertex. The new graph also has 3 identical homogeneous cores (the wheel graph $W_7$ which is described at the top of Figure \ref{figure2}(c)) which are induced by dashed edges ( those with probability $\frac{1}{2}$).  We shrink each of these homogeneous cores to a single vertex producing the triangle graph $C_3$ as in the bottom of Figure \ref{figure2}(c)). Since the triangle graph $C_3$ is homogeneous, we shrink it to a single vertex.

Next, we describe the deflation process by decomposing a graph recursively using the Beurling partition $P_{ max}$.  Starting from a connected graph $G$. Let \[P_{ max}= \left\{ V_1,V_2,...,V_{k_{P_{ max}}} \right\}\] be the Beurling partition of $G$ constructed as in Theorem \ref{etamax}.  Let $G_i = (V_i,E_i)$ be the subgraph induced by $V_i$ for $i=1,2,\dots,k_{P_{ max}}$. By Theorem \ref{serialmod} and Corollary \ref{shrunkhom}, we have

\[ \MEO(\Ga_{G_{P_{ max}}}) = \frac{(k_{P_{ max}}-1)^2}{|E_{P_{ max}}|},\]
so
\[ \MEO (\Ga_G) = \MEO (\Ga_{G_1}) + \MEO (\Ga_{G_2}) +\dots + \MEO (\Ga_{G_{k_{P_{ max}}}})+ \frac{(k_{P_{ max}}-1)^2}{|E_{P_{ max}}|}.\]

Iterating this procedure on each $G_i$ gives a decomposition of the graph $G$ into a sequence of shrunk graphs and subgraphs, eventually terminating when all final subgraphs are homogeneous and can not be decomposed further. By tracing back through the deflation process and using Theorem \ref{serialmod}, we can construct an optimal pmf for the $\MEO$ problem on $G$ from corresponding optimal pmfs on shrunk graphs and subgraphs during the deflation process. 

\begin{figure}[t]
	\centering
	
	\begin{center}
		\begin{tabular}{ |c|c|c| } 
			\hline
			(a) &  (b) & (c) \\
			\begin{tikzpicture}[scale=1.3, auto, node distance=3cm,  thin]
				\begin{scope}[every node/.style={circle,draw=black,fill=black!100!,font=\sffamily\Large\bfseries}]
					
					\node (A) [scale=0.3]at (0,0) {};
					\node (B)[scale=0.3] at (1,0) {};
					\node (C)[scale=0.3] at (1.5,0.87) {};
					\node (D) [scale=0.3]at (1,1.73) {};
					\node (E)[scale=0.3] at (0,1.73) {};
					\node (F)[scale=0.3] at (-0.5,0.87) {};
					\node (G)[scale=0.3] at (0.25,0.43) {};
					\node (H) [scale=0.3]at (0.75,0.43) {};
					\node (I)[scale=0.3] at (1,0.87) {};
					\node (J)[scale=0.3] at (0.75,1.3) {};
					\node (K)[scale=0.3] at (0.25,1.3) {};
					\node (L)[scale=0.3] at (0,0.87) {};
					
					\node (A1) [scale=0.3]at (0+3,0) {};
					\node (B1)[scale=0.3] at (1+3,0) {};
					\node (C1)[scale=0.3] at (1.5+3,0.87) {};
					\node (D1) [scale=0.3]at (1+3,1.73) {};
					\node (E1)[scale=0.3] at (0+3,1.73) {};
					\node (F1)[scale=0.3] at (-0.5+3,0.87) {};
					\node (G1)[scale=0.3] at (0.25+3,0.43) {};
					\node (H1) [scale=0.3]at (0.75+3,0.43) {};
					\node (I1)[scale=0.3] at (1+3,0.87) {};
					\node (J1)[scale=0.3] at (0.75+3,1.3) {};
					\node (K1)[scale=0.3] at (0.25+3,1.3) {};
					\node (L1)[scale=0.3] at (0+3,0.87) {};
					
					\node (A2) [scale=0.3]at (0+1.5,0-2.6) {};
					\node (B2)[scale=0.3] at (1+1.5,0-2.6) {};
					\node (C2)[scale=0.3] at (1.5+1.5,0.87-2.6) {};
					\node (D2) [scale=0.3]at (1+1.5,1.73-2.6) {};
					\node (E2)[scale=0.3] at (0+1.5,1.73-2.6) {};
					\node (F2)[scale=0.3] at (-0.5+1.5,0.87-2.6) {};
					\node (G2)[scale=0.3] at (0.25+1.5,0.43-2.6) {};
					\node (H2) [scale=0.3]at (0.75+1.5,0.43-2.6) {};
					\node (I2)[scale=0.3] at (1+1.5,0.87-2.6) {};
					\node (J2)[scale=0.3] at (0.75+1.5,1.3-2.6) {};
					\node (K2)[scale=0.3] at (0.25+1.5,1.3-2.6) {};
					\node (L2)[scale=0.3] at (0+1.5,0.87-2.6) {};
					
				\end{scope}
				\begin{scope}[every edge/.style={draw=black,thin}]
					
					\draw  (G) edge node{} (H); 
					\draw  (H) edge node{} (I); 
					\draw  (I) edge node{} (J);
					\draw  (J) edge node{} (K);
					\draw  (K) edge node{} (L); 
					\draw  (L) edge node{} (G);

					\draw  (G1) edge node{} (H1); 
					\draw  (H1) edge node{} (I1); 
					\draw  (I1) edge node{} (J1);
					\draw  (J1) edge node{} (K1);
					\draw  (K1) edge node{} (L1); 
					\draw  (L1) edge node{} (G1);
					
					\draw  (G2) edge node{} (H2); 
					\draw  (H2) edge node{} (I2); 
					\draw  (I2) edge node{} (J2);
					\draw  (J2) edge node{} (K2);
					\draw  (K2) edge node{} (L2); 
					\draw  (L2) edge node{} (G2); 
					
					\draw  (G) edge node{} (I);
					\draw  (G) edge node{} (J);
					\draw  (G) edge node{} (K);
					\draw  (H) edge node{} (J);
					\draw  (H) edge node{} (K);
					\draw  (H) edge node{} (L);
					\draw  (I) edge node{} (K);
					\draw  (I) edge node{} (L);
					\draw  (J) edge node{} (L);

					\draw  (G1) edge node{} (I1);
					\draw  (G1) edge node{} (J1);
					\draw  (G1) edge node{} (K1);
					\draw  (H1) edge node{} (J1);
					\draw  (H1) edge node{} (K1);
					\draw  (H1) edge node{} (L1);
					\draw  (I1) edge node{} (K1);
					\draw  (I1) edge node{} (L1);
					\draw  (J1) edge node{} (L1);
					
					\draw  (G2) edge node{} (I2);
					\draw  (G2) edge node{} (J2);
					\draw  (G2) edge node{} (K2);
					\draw  (H2) edge node{} (J2);
					\draw  (H2) edge node{} (K2);
					\draw  (H2) edge node{} (L2);
					\draw  (I2) edge node{} (K2);
					\draw  (I2) edge node{} (L2);
					\draw  (J2) edge node{} (L2);
					
				\end{scope}
				
				\begin{scope}[every edge/.style={draw=black,dashed}]
					
					\draw  (A) edge node{} (B);
					\draw  (B) edge node{} (C);
					\draw  (C) edge node{} (D);
					\draw  (D) edge node{} (E);
					\draw  (E) edge node{} (F);
					\draw  (F) edge node{} (A);
					
					\draw  (A1) edge node{} (B1);
					\draw  (B1) edge node{} (C1);
					\draw  (C1) edge node{} (D1);
					\draw  (D1) edge node{} (E1);
					\draw  (E1) edge node{} (F1);
					\draw  (F1) edge node{} (A1);
					
					\draw  (A2) edge node{} (B2);
					\draw  (B2) edge node{} (C2);
					\draw  (C2) edge node{} (D2);
					\draw  (D2) edge node{} (E2);
					\draw  (E2) edge node{} (F2);
					\draw  (F2) edge node{} (A2);
					
					\draw  (A) edge node{} (G);
					\draw  (B) edge node{} (H);
					\draw  (C) edge node{} (I);
					\draw  (D) edge node{} (J);
					\draw  (E) edge node{} (K);
					\draw  (F) edge node{} (L);
					
					\draw  (A1) edge node{} (G1);
					\draw  (B1) edge node{} (H1);
					\draw  (C1) edge node{} (I1);
					\draw  (D1) edge node{} (J1);
					\draw  (E1) edge node{} (K1);
					\draw  (F1) edge node{} (L1);
					
					\draw  (A2) edge node{} (G2);
					\draw  (B2) edge node{} (H2);
					\draw  (C2) edge node{} (I2);
					\draw  (D2) edge node{} (J2);
					\draw  (E2) edge node{} (K2);
					\draw  (F2) edge node{} (L2);
					
				\end{scope}
				
				\begin{scope}[every edge/.style={draw=black,thick, dotted}]
					
					\draw  (C) edge node{} (F1);
					\draw  (B) edge node{} (E2);
					\draw  (D2) edge node{} (A1);
				\end{scope}
			\end{tikzpicture}
			& 
			\begin{tikzpicture}[scale=1.3, auto, node distance=3cm,  thin]
				\begin{scope}[every node/.style={circle,draw=black,fill=black!100!,font=\sffamily\Large\bfseries}]
					
					\node (A) [scale=0.3]at (0,0) {};
					\node (B)[scale=0.3] at (1,0) {};
					\node (C)[scale=0.3] at (1.5,0.87) {};
					\node (D) [scale=0.3]at (1,1.73) {};
					\node (E)[scale=0.3] at (0,1.73) {};
					\node (F)[scale=0.3] at (-0.5,0.87) {};
					\node (G)[scale=0.3] at (0.25,0.43) {};
					\node (H) [scale=0.3]at (0.75,0.43) {};
					\node (I)[scale=0.3] at (1,0.87) {};
					\node (J)[scale=0.3] at (0.75,1.3) {};
					\node (K)[scale=0.3] at (0.25,1.3) {};
					\node (L)[scale=0.3] at (0,0.87) {};
					
					\node (A1) [scale=0.3]at (0+3,0) {};
					\node (B1)[scale=0.3] at (1+3,0) {};
					\node (C1)[scale=0.3] at (1.5+3,0.87) {};
					\node (D1) [scale=0.3]at (1+3,1.73) {};
					\node (E1)[scale=0.3] at (0+3,1.73) {};
					\node (F1)[scale=0.3] at (-0.5+3,0.87) {};
					\node (G1)[scale=0.3] at (0.25+3,0.43) {};
					\node (H1) [scale=0.3]at (0.75+3,0.43) {};
					\node (I1)[scale=0.3] at (1+3,0.87) {};
					\node (J1)[scale=0.3] at (0.75+3,1.3) {};
					\node (K1)[scale=0.3] at (0.25+3,1.3) {};
					\node (L1)[scale=0.3] at (0+3,0.87) {};
					
					\node (A2) [scale=0.3]at (0+1.5,0-2.6) {};
					\node (B2)[scale=0.3] at (1+1.5,0-2.6) {};
					\node (C2)[scale=0.3] at (1.5+1.5,0.87-2.6) {};
					\node (D2) [scale=0.3]at (1+1.5,1.73-2.6) {};
					\node (E2)[scale=0.3] at (0+1.5,1.73-2.6) {};
					\node (F2)[scale=0.3] at (-0.5+1.5,0.87-2.6) {};
					\node (G2)[scale=0.3] at (0.25+1.5,0.43-2.6) {};
					\node (H2) [scale=0.3]at (0.75+1.5,0.43-2.6) {};
					\node (I2)[scale=0.3] at (1+1.5,0.87-2.6) {};
					\node (J2)[scale=0.3] at (0.75+1.5,1.3-2.6) {};
					\node (K2)[scale=0.3] at (0.25+1.5,1.3-2.6) {};
					\node (L2)[scale=0.3] at (0+1.5,0.87-2.6) {};
				\end{scope}
				
				\begin{scope}[every node/.style={circle,draw=black,font=
						\sffamily\Large\bfseries}]
					
					\node (W1) [scale=0.3]at (2,-0.43) {};
					\node (U1) [scale=0.3]at (2.38,0.22) {};
					\node (V1) [scale=0.3]at (1.63,0.22) {};
				\end{scope}

				\begin{scope}[every edge/.style={draw=black,thin}]
					
					\draw  (G) edge node{} (H); 
					\draw  (H) edge node{} (I); 
					\draw  (I) edge node{} (J);
					\draw  (J) edge node{} (K);
					\draw  (K) edge node{} (L); 
					\draw  (L) edge node{} (G);

					\draw  (G1) edge node{} (H1); 
					\draw  (H1) edge node{} (I1); 
					\draw  (I1) edge node{} (J1);
					\draw  (J1) edge node{} (K1);
					\draw  (K1) edge node{} (L1); 
					\draw  (L1) edge node{} (G1);
					
					\draw  (G2) edge node{} (H2); 
					\draw  (H2) edge node{} (I2); 
					\draw  (I2) edge node{} (J2);
					\draw  (J2) edge node{} (K2);
					\draw  (K2) edge node{} (L2); 
					\draw  (L2) edge node{} (G2); 
					
					\draw  (G) edge node{} (I);
					\draw  (G) edge node{} (J);
					\draw  (G) edge node{} (K);
					\draw  (H) edge node{} (J);
					\draw  (H) edge node{} (K);
					\draw  (H) edge node{} (L);
					\draw  (I) edge node{} (K);
					\draw  (I) edge node{} (L);
					\draw  (J) edge node{} (L);

					\draw  (G1) edge node{} (I1);
					\draw  (G1) edge node{} (J1);
					\draw  (G1) edge node{} (K1);
					\draw  (H1) edge node{} (J1);
					\draw  (H1) edge node{} (K1);
					\draw  (H1) edge node{} (L1);
					\draw  (I1) edge node{} (K1);
					\draw  (I1) edge node{} (L1);
					\draw  (J1) edge node{} (L1);
					
					\draw  (G2) edge node{} (I2);
					\draw  (G2) edge node{} (J2);
					\draw  (G2) edge node{} (K2);
					\draw  (H2) edge node{} (J2);
					\draw  (H2) edge node{} (K2);
					\draw  (H2) edge node{} (L2);
					\draw  (I2) edge node{} (K2);
					\draw  (I2) edge node{} (L2);
					\draw  (J2) edge node{} (L2);
					
				\end{scope}
				
				\begin{scope}[every edge/.style={draw=black,dashed}]
					
					\draw  (A) edge node{} (B);
					\draw  (B) edge node{} (C);
					\draw  (C) edge node{} (D);
					\draw  (D) edge node{} (E);
					\draw  (E) edge node{} (F);
					\draw  (F) edge node{} (A);
					
					\draw  (A1) edge node{} (B1);
					\draw  (B1) edge node{} (C1);
					\draw  (C1) edge node{} (D1);
					\draw  (D1) edge node{} (E1);
					\draw  (E1) edge node{} (F1);
					\draw  (F1) edge node{} (A1);
					
					\draw  (A2) edge node{} (B2);
					\draw  (B2) edge node{} (C2);
					\draw  (C2) edge node{} (D2);
					\draw  (D2) edge node{} (E2);
					\draw  (E2) edge node{} (F2);
					\draw  (F2) edge node{} (A2);
					
					\draw  (A) edge node{} (G);
					\draw  (B) edge node{} (H);
					\draw  (C) edge node{} (I);
					\draw  (D) edge node{} (J);
					\draw  (E) edge node{} (K);
					\draw  (F) edge node{} (L);
					
					\draw  (A1) edge node{} (G1);
					\draw  (B1) edge node{} (H1);
					\draw  (C1) edge node{} (I1);
					\draw  (D1) edge node{} (J1);
					\draw  (E1) edge node{} (K1);
					\draw  (F1) edge node{} (L1);
					
					\draw  (A2) edge node{} (G2);
					\draw  (B2) edge node{} (H2);
					\draw  (C2) edge node{} (I2);
					\draw  (D2) edge node{} (J2);
					\draw  (E2) edge node{} (K2);
					\draw  (F2) edge node{} (L2);
					
				\end{scope}
				
				\begin{scope}[every edge/.style={draw=black,thick, dotted}]
					
					\draw  (W1) edge node{} (U1);
					\draw  (U1) edge node{} (V1);
					\draw  (V1) edge node{} (W1);
				\end{scope}
			\end{tikzpicture} &  
			\begin{tikzpicture}[scale=1.3, auto, node distance=3cm,  thin]
				\begin{scope}[every node/.style={circle,draw=black,font=
						\sffamily\Large\bfseries}]
					\node (A) [scale=0.3]at (0,0) {};
					\node (B)[scale=0.3] at (1,0) {};
					\node (C)[scale=0.3] at (1.5,0.87) {};
					\node (D) [scale=0.3]at (1,1.73) {};
					\node (E)[scale=0.3] at (0,1.73) {};
					\node (F)[scale=0.3] at (-0.5,0.87) {};
					
					\node (U)[scale=0.3] at (0.5,0.87) {};
				\end{scope}
				\begin{scope}[every node/.style={circle,draw=black,fill=black!100!,font=\sffamily\Large\bfseries}]
					
					\node (G1)[scale=0.3] at (0.25,-2.17) {};
					\node (H1) [scale=0.3]at (0.75,-2.17) {};
					\node (I1)[scale=0.3] at (1,-1.73) {};
					\node (J1)[scale=0.3] at (0.75,-1.3) {};
					\node (K1)[scale=0.3] at (0.25,-1.3) {};
					\node (L1)[scale=0.3] at (0,-1.73) {};
					
					\node (A1) [scale=0.3]at (0,0-2.6) {};
					\node (B1)[scale=0.3] at (1,0-2.6) {};
					\node (C1)[scale=0.3] at (1.5,0.87-2.6) {};
					\node (D1) [scale=0.3]at (1,1.73-2.6) {};
					\node (E1)[scale=0.3] at (0,1.73-2.6) {};
					\node (F1)[scale=0.3] at (-0.5,0.87-2.6) {};

				\end{scope}
				\begin{scope}[every edge/.style={draw=black,thin}]
					
					\draw  (G1) edge node{} (H1); 
					\draw  (H1) edge node{} (I1); 
					\draw  (I1) edge node{} (J1);
					\draw  (J1) edge node{} (K1);
					\draw  (K1) edge node{} (L1); 
					\draw  (L1) edge node{} (G1);
					
					\draw  (G1) edge node{} (I1);
					\draw  (G1) edge node{} (J1);
					\draw  (G1) edge node{} (K1);
					\draw  (H1) edge node{} (J1);
					\draw  (H1) edge node{} (K1);
					\draw  (H1) edge node{} (L1);
					\draw  (I1) edge node{} (K1);
					\draw  (I1) edge node{} (L1);
					\draw  (J1) edge node{} (L1);
				\end{scope}
				
				\begin{scope}[every edge/.style={draw=black,dashed}]
					\draw  (A) edge node{} (B);
					\draw  (B) edge node{} (C);
					\draw  (C) edge node{} (D);
					\draw  (D) edge node{} (E);
					\draw  (E) edge node{} (F);
					\draw  (F) edge node{} (A);

					\draw  (A) edge node{} (U);
					\draw  (B) edge node{} (U);
					\draw  (C) edge node{} (U);
					\draw  (D) edge node{} (U);
					\draw  (E) edge node{} (U);
					\draw  (F) edge node{} (U);
				\end{scope}
			\end{tikzpicture}
			\\ 
			\hline
		\end{tabular}
	\end{center}

	\caption{Deflation process using the Beurling partition $P_{max}$. Edges are styled according to the optimal density $\eta^*$. Solid edges are present with $\eta^* = \frac{1}{3}$, dashed edges with $\eta^* = \frac{1}{2}$ and dotted edges with  $\eta^* = \frac{2}{3}$.}\label{figure1}
\end{figure}
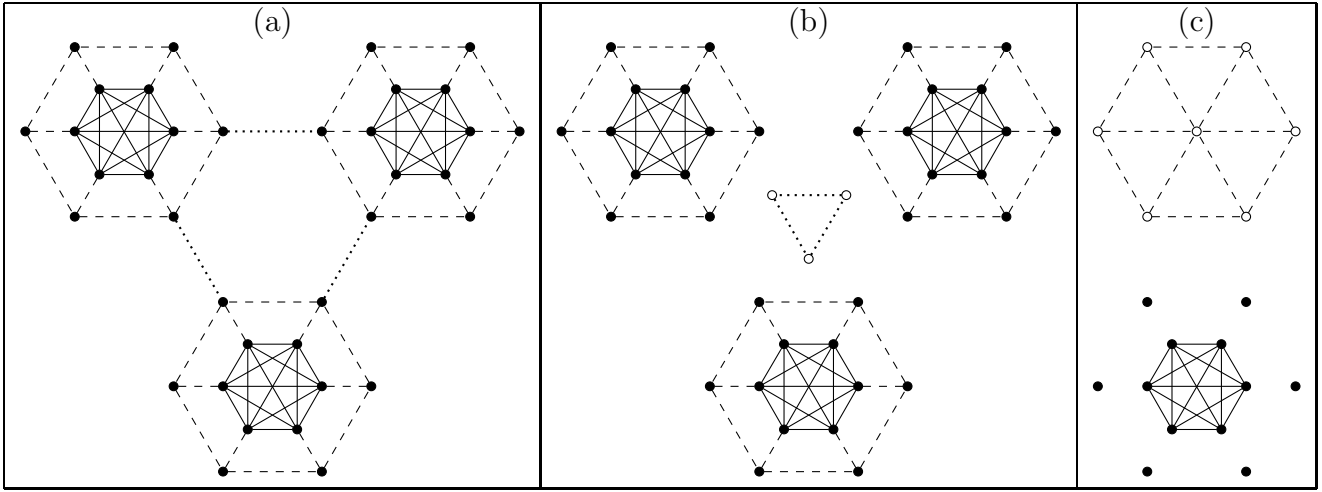

Figure \ref{figure1} describes an example of the deflation process using the Beurling partition $P_{ max}$. Consider a graph $G$ in Figure \ref{figure1}(a). Here, the optimal edge usage probability function $\eta^{*}(e)$ takes three distinct values, indicated by the edge styles as described in the figure caption. 
First, we remove the dotted edges from $G$. As the result, we obtain three identical connected components $G_1, G_2, G_3$ which are three identical subgraphs of $G$. Then, in the graph $G$, we shrink each $G_i$ to a vertex meaning that we identify all the vertices in each $G_i$ as a single vertex and remove all resulting self-loops. Then, the resulting shrunk graph is the triangle graph $C_3$ as in Figure \ref{figure1}(b). Notice that the edge set of $G$ is the disjoint union of the edge sets of $G_1,G_2,G_3$ and $C_3$. So, we say that $G$ decomposes into three subgraphs $G_1,G_2,G_3$ and the triangle graph $C_3$ as in Figure \ref{figure1}(b). For each $G_i$, we will continue to decompose it. We remove the dashed edges from $G_i$. Then, we obtain $7$ connected components which include $6$ isolated vertices and the complete graph $K_6$ as in Figure \ref{figure1}(c). In $G_i$, we shrink that complete graph $K_6$ to a single vertex to construct the shrunk graph which is, in this case, the wheel graph $W_7$. That means each $G_i$ decomposes into 6 isolated vertices, the complete graph $K_6$ and the wheel graph $W_7$ as in Figure \ref{figure1}(c). The deflation process stops here since all the created subgraphs (isolated vertices, $C_3$, $K_6$ and $W_7$) are homogeneous.

\section{Fulkerson blocker family of the spanning tree family}\label{sec:ful}

In this section, we present an alternative approach to determining the Fulkerson blocker for spanning tree families, distinct from the method in \cite{chopraon}.
\begin{theorem}\label{theo:fulkerson}
	Let $G=(V,E)$ be a connected, unweighted, undirected multigraph with no self-loops. Let $\Gamma$ be the spanning tree family of $G$ with usage vectors given by the indicator functions. Let $\Phi$ be the set of all feasible partitions of $G$. For each feasible partition $P= \left\{ V_1,V_2,...,V_{k_P} \right\}$ with $k_P \geq 2$, let $E_P$ be the cut set of $P$ which is set of all edges in $G$ that connect vertices that belong to different $V_i$. Then, the Fulkerson blocker family $\widehat{\Gamma}$ is the set of all
	vectors 
	\begin{align*}
		\frac{1}{k_P-1}\mathbbm{1}_{E_P},
	\end{align*}
	for all $P \in \Phi$ whose shrunk graph $G_P$ is vertex-biconnected.
\end{theorem}

Let $\Theta \subset \Phi$ be the family of all feasible partitions $P$ whose shrunk graph $G_P$ is vertex-biconnected, with usage given as in (\ref{usage2}). In Section \ref{sec:firsth}, we reprove the statement (\ref{fact:1}) using a different approach, which is based on a result of Nash-Williams and Tutte \cite{edge-disjoint,on}. For completeness, in Section \ref{sec:secondhalf}, we reprove the following statements in detail
\begin{equation}\label{fact:2}
	\Theta \subset \Gahat,
	\qquad\text{and}\qquad
	\Phi \cap \Gahat \subset \Theta,
\end{equation}
using ideas from \cite{chopraon}, but in the language of extreme points. Finally, using (\ref{fact:1}) and (\ref{fact:2}), we obtain (\ref{fact:0}).

\subsection{The strength of graphs and 1-modulus} \label{sec:firsth}
Let $G= (V,E,\si)$ be a weighted connected multigraph with edge weights $\si \in \R^E_{>0}$. Let $\Ga = \Ga_G$ be the spanning tree family of $G$.  Let $\Phi$ be the family of all feasible partitions $P$ of $G$, which we identify with the set of usage vectors in $\R_{\ge 0}^E$, as defined in (\ref{usage2}). Let $\Gahat$ be the Fulkerson blocker dual family for $\Ga$. In this section,  we establish the equality between the strength of $G$ and the 1-modulus of $\Ga$, using a result of Nash-Williams and Tutte \cite{edge-disjoint}, \cite{on}. With this connection, we give a different proof for the claim in Equation (\ref{fact:1}).
\begin{claim} \label{claim:1} We have  $\widehat{\Gamma} \subset \Phi$. 
\end{claim} 
We first recall the 1-modulus problem $\Mod_{1,\si}(\Ga)$:

\begin{equation}
	\begin{array}{ll}
		\underset{\rho \in \R^{E}}{\text{minimize}}    &\si^T\rho \\
		\text{subject to } &\sum\limits_{e \in E} \cN(\ga,e)\rho(e) \geq 1, \quad \forall \ga \in \Ga; \\
		&\rho \geq 0.  	 
	\end{array}
\end{equation}
The dual problem of $\Mod_{1,\si}(\Ga)$ is described in \cite{modulus} as follows,
\begin{equation}  \label{dualrrr}
	\begin{array}{ll}
		\underset{\lambda \in \R^{\Ga}}{\text{maximize}}    &\lambda^T\mathbf{1} \\
		\text{subject to } &\sum\limits_{e \in \ga} \lambda(\ga) \leq \sigma(e), \quad \forall e \in E; \\
		&\lambda \geq 0, 	 
	\end{array}
\end{equation}
where $\one$ is the column vector of all ones.
\begin{remark}\label{remarkdis}
	When $\si(e)=1$ for all $e \in E$ and $\lambda \in \Z^{\Ga}$, the value of the dual problem (\ref{dualrrr}) is equal to the maximum number of disjoint spanning trees in the graph $G$.
\end{remark}
We recall the strength of the graph $G$, is given by
$
S_{\si}(G) = \min\limits_{P \in \Phi} \frac{\si(E_P)}{k_P-1}$.

\begin{remark}\label{rsi} The strength of $G$ is $1$-homogeneous with respect to $\si$, meaning that $S_{r\si}(G)=rS_{\si}(G)$ for any positive real number $r$.
\end{remark}

\begin{lemma}\label{continuity}
	Let $G= (V,E,\si)$ be a weighted connected graph with edge weights $\si \in \R^E_{>0}$. Then, the function $\si \mapsto S_{\si}(G)$ is Lipschitz continuous and increasing along the coordinate directions.
\end{lemma}

\begin{proof}
	Let $ \si_1, \si_2\in \R^E_{>0}$. Let $P_1$ and $P_2$ be  critical partitions for $S_{\si_1}(G)$ and $S_{\si_2}(G)$, respectively. Without loss of generality, we assume that $S_{\si_1}(G) \leq S_{\si_2}(G)$. Then,
	
	\begin{align*}
		S_{\si_2}(G) - S_{\si_1}(G) &\leq \frac{\si_2(E_{P_1})}{k_{P_1}-1} - \frac{\si_1(E_{P_1})}{k_{P_1}-1} = \frac{1}{k_{P_1}-1}\sum\limits_{e \in E_{P_1}} (\si_2(e)-\si_1(e)) \\
		&   \leq  \frac{1}{k_{P_1}-1} \sum\limits_{e \in E_{P_1}} |\si_2(e)-\si_1(e)|\\
		&\leq \sum\limits_{e \in E_{P_1}} |\si_2(e)-\si_1(e)| \leq \sum\limits_{e \in E} |\si_2(e)-\si_1(e)|  = \Vert \si_2-\si_1 \Vert_1.
	\end{align*}
	where $\Vert \cdot \Vert_1$ is the 1-norm on $\R^E$. Note that all norms on $\R^E$ are equivalent.
	
	Assume that $\si \leq \si',$ then $\si(A) \leq \si'(A)$ for all $A \subset E.$ Hence, 
	\[ \frac{\si(E_{P})}{k_{P}-1} \leq \frac{\si'(E_{P})}{k_{P}-1}, \] for any feasible partition $P$. Therefore, $S_{\si}(G) \leq S_{\si'}(G)$. In conclusion, the function $\si \mapsto S_{\si}(G)$ is Lipschitz continuous and increasing along the coordinate directions.
\end{proof}
The following theorem shows that $S_{\si}(G)$ is the same as  $\Mod_{1,\si}(\Ga)$. 
\begin{proposition}\label{sm}
	Let $G= (V,E,\si)$ be a weighted connected graph with edge weights $\si \in \R^E_{>0}$. Let $\Ga = \Ga_G$ be the spanning tree family of $G$. Let $S_{\si}(G)$ be the strength of $G$. Then,
	
	\begin{equation}
		S_{\si}(G) = \Mod_{1,\si}(\Ga).
	\end{equation}
\end{proposition}
By the continuity of  $\Mod_{1,\si}(\Ga)$ \cite[Lemma 6.4]{modulus} and $S_{\si}(G)$ (Lemma \ref{continuity}) with respect to $\si$,  it is enough to  prove Proposition \ref{sm} in the case where $\si \in \Q^E$. By homogeneity of both $\Mod_{1,\si}(\Ga)$ and $S_{\si}(G)$, we may consider the case $\si \in \Z^E$.
First, we state a result of Tutte \cite{on}, this result was independently
proved by Nash-Williams \cite{edge-disjoint}.

\begin{theorem}[\cite{edge-disjoint,on}] \label{TutteN} Let $G=(V,E)$ be a connected graph. Then, $G$ has $k$ disjoint spanning trees if and only if, for any partition $P = \left\{ V_1,V_2,...,V_{k_P} \right\}$ of the node set $V$, with $k_P \geq 2$ and with cut-set $E_P $, we have
	\begin{equation}\label{ktrees2}
		\frac{|E_P|}{k_P-1} \geq k.
	\end{equation}
\end{theorem}
Since we are interested in feasible partitions, we want to slightly modify Theorem \ref{TutteN}. In particular, Theorem \ref{TutteN} can be restricted to feasible partitions as follows.
\begin{proposition} \label{Tutte}
	A graph $G=(V,E)$  has $k$ disjoint spanning trees if and only if $S(G) \geq k.$ 
\end{proposition}
\begin{proof}
	Assume that $G$ has $k$ disjoint spanning trees. By Theorem \ref{TutteN}, (\ref{ktrees2}) holds for any  partition $P$. In particular, it holds for any feasible one.
	
	Assume that (\ref{ktrees2}) holds for any feasible partition $P$. Let $P= \left\{ V_1,V_2,...,V_{k_P} \right\}$ be an arbitrary partition of the node set $V$ with $k_P \geq 2$. For each subgraph $G(V_i)$ of $G$ induced by $V_i$, let $m(i)$ be the number of connected components of $G(V_i)$ and let $V_{i_1},V_{i_2},\dots,V_{i_{m(i)}}$ be the sets of nodes of those connected components. Consider the partition $P' := \left\{V_{i_j}:1\leq i \leq k_P, 1\leq j \leq m(i) \right\}$. Then, $k_{P'} \geq k_{P} $. Since each $V_{i_j}$ is connected,  the partition $P'$ is feasible. Notice that $E_P$ is the set of edges connecting nodes between $V_{i_j}$ with different $i$ and there are no edges connecting nodes between $V_{i_j}$ with the same $i$ and different $j$, therefore we have $E_{P'}=E_{P}$.
	This implies 
	\[\frac{|E_P|}{k_P-1} \geq \frac{|E_{P'}|}{k_{P'}-1} \geq k.\]
	Hence, (\ref{ktrees2}) holds for any partition $P$. By Theorem \ref{TutteN},  $G$ has $k$ disjoint spanning trees.
	
\end{proof}

\begin{corollary}\label{coro1}
	Let  $G= (V,E)$ be a connected graph. Let $S(G)$ be the strength of $G$. Then, the maximum number of disjoint spanning trees in the graph $G$ is $\lfloor S(G) \rfloor$, where $\lfloor \cdot \rfloor$ is the floor function.
\end{corollary}

\begin{lemma}\label{smz}
	Let $G= (V,E,\si)$ be a weighted connected graph with edge weights $\si \in \Z^E_{>0}$.  Then, $\lfloor S_{\si}(G) \rfloor$ is the value of the following problem.
	\begin{equation} \label{dualzzz}
		\begin{array}{ll}
			\underset{\lambda \in \Z^{\Ga}}{\text{\rm maximize}}    &\lambda^T\mathbf{1} \\
			\text{\rm subject to } &\sum\limits_{e \in \ga} \lambda(\ga) \leq \sigma(e), \qquad \forall e \in E; \\
			&\lambda \geq 0.  	 
		\end{array}
	\end{equation}
\end{lemma}
Note that the problem in (\ref{dualzzz}) is similar to the dual modulus problem in (\ref{dualr}), but with $\la$ restricted to be integer-valued.
\begin{proof}[Proof of Lemma \ref{smz}]
	We create an unweighted multigraph $G'$ with node set $V$ from $G$ by making $\si(e)$ copies of each edge $e \in E.$ Then, $S_{\si}(G) = S(G')$. By Corollary \ref{coro1},  the maximum number of disjoint spanning trees of $G'$ is  $ \lfloor S(G') \rfloor =\lfloor S_{\si}(G) \rfloor$. Notice that the problem (\ref{dualzzz}) on $G$ is equivalent to the same problem on $G'$ but with $\sigma\equiv 1$.
	Hence, the value of  (\ref{dualzzz}) is equal to the maximum number of disjoint spanning trees of $G'$ by Remark \ref{remarkdis}.
\end{proof}

\begin{proof}[Proof of Proposition \ref{sm}]
	Assume that $\si \in \Q^E$. Since, the dual problem (\ref{dualr}) is a linear program, there exists an optimal solution $\lambda^{\ast} \in \Q^{\Ga}$. In particular, $M := \Mod_{1,\si}(\Ga)$ is the value of the following problem
	\begin{equation}\label{eq:dualq}
		{\rm maximize} \lbr \lambda^T\mathbf{1}: \la \in  \Q^{\Ga}, \sum\limits_{e \in \ga} \lambda(\ga) \leq \sigma(e) \hspace{2pt} \forall e \in E, \lambda \geq 0\rbr ,
	\end{equation}
	and
	\begin{equation}\label{eq:dualq-opt}
		\lambda^{\ast} \in \argmax \lbr \lambda^T\mathbf{1}: \la \in  \Q^{\Ga}, \sum\limits_{e \in \ga} \lambda(\ga) \leq \sigma(e) \hspace{2pt} \forall e \in E, \lambda \geq 0\rbr .
	\end{equation}
	Since $\si \in \Q^E$ and $\lambda^* \in \Q^{\Ga}$, there exists a positive integer $h$ such that
	\[\delta(e) := h\sigma(e) \in \Z, \qquad\forall e \in E \] and \[\alpha^\ast(\ga) := h\lambda^\ast(\ga) \in \Z, \qquad\forall \ga \in \Ga.\]
	Perform the change of variable $\alpha = h\lambda$ in (\ref{eq:dualq}) and (\ref{eq:dualq-opt}) to obtain
	\begin{align*}
		hM = \max \lbr \alpha^T\mathbf{1}: \alpha \in  \Q^{\Ga}, \sum\limits_{e \in \ga} \alpha(\ga) \leq \delta(e) \hspace{2pt} \forall e \in E, \alpha \geq 0\rbr ,
	\end{align*}
	\begin{align*}
		\alpha^{\ast} \in \argmax \lbr \alpha^T\mathbf{1}: \alpha \in  \Q^{\Ga}, \sum\limits_{e \in \ga} \alpha(\ga) \leq \delta(e) \hspace{2pt} \forall e \in E, \alpha \geq 0\rbr .
	\end{align*}
	And, since $\alpha^{\ast} \in \Z^{\Ga}$, we have
	\[hM = \max \lbr \alpha^T\mathbf{1}: \alpha \in  \Z^{\Ga}, \sum\limits_{e \in \ga} \alpha(\ga) \leq \delta(e) \hspace{2pt} \forall e \in E, \alpha \geq 0\rbr ,\]
	
	\[ \alpha^{\ast} \in \argmax \lbr \alpha^T\mathbf{1}: \alpha \in  \Z^{\Ga}, \sum\limits_{e \in \ga} \alpha(\ga) \leq \delta(e) \hspace{2pt} \forall e \in E, \alpha \geq 0\rbr .\]
	Let $P^{\ast}$ be a critical partition of $G =(V,E,\delta)$. Denote $\kappa := k_{P^{\ast}}-1$.
	Then, by Remark \ref{rsi}, we have
	\[S_{\kappa\delta}(G)= \kappa S_{\delta}(G)= \kappa\frac{\delta(E_{P^{\ast}})}{k_{P^{\ast}}-1} = \delta(E_{P^{\ast}}) \in \Z.\]
	Apply Lemma \ref{smz} to the graph $G=(V,E,\kappa \delta)$, we have that

	\[S_{\kappa\delta}(G) = \max \lbr \beta^T\mathbf{1}: \beta \in  \Z^{\Ga}, \sum\limits_{e \in \ga} \beta(\ga) \leq \kappa \delta(e) = \kappa h\sigma(e) \hspace{2pt} \forall e \in E, \beta \geq 0\rbr. \]
	As before, perform another change of variables, $\beta = \kappa h\lambda$,  for the dual problem (\ref{eq:dualq}). We obtain
	\[\kappa hM = \max \lbr \beta^T\mathbf{1}: \beta \in  \Z^{\Ga}, \sum\limits_{e \in \ga} \beta(\ga) \leq \kappa\delta(e) = \kappa h\sigma(e) \hspace{2pt} \forall e \in E, \beta \geq 0\rbr. \]
	Therefore, 
	\begin{equation*}
		\kappa hS_{\si}(G)= S_{\kappa \delta}(G) = \kappa hM.
	\end{equation*}
	Hence,
	\[ S_{\si}(G)= \Mod_{1,\si}(\Ga).\]
	The proof is completed.
\end{proof}
Before proving Claim \ref{claim:1}, we make the following remark. Let $P$ be any feasible partition. Since  $G$ is connected, the restriction of any spanning tree $\ga$ of $G$ onto $E_P$ forms a connected spanning subgraph of the shrunk graph $G_P$. Equivalently, $|E_P\cap \ga| \geq k_P-1$.  Therefore, we always have
\begin{equation*}
	\frac{1}{k_P-1}\mathbbm{1}_{E_P}  \in \Adm(\Ga).
\end{equation*}
\begin{proof}[Proof of Claim \ref{claim:1}]
	By Lemma \ref{lem:vertex},  for any $\widehat{\gamma} \in \widehat{\Gamma}$, we choose $\sigma \in \mathbb{R}^E_{> 0}$ such that $\widehat{\gamma}$ is the unique solution for both problems in (\ref{mod:mm}).
	By Proposition \ref{sm}, we have
	\[ \min\limits_{\widehat{\ga} \in \widehat{\Ga}}\sigma^{T}\widehat{\ga} = \Mod_{1,\sigma}(\Gamma) = S_\si(G). \]
	
	Let $P$ be a critical partition for $S_{\si}(G)$. Then, $\frac{1}{k_P-1}\mathbbm{1}_{E_P} \in \Adm(\Ga)$ and, denoting the dot product in $\R^E$ with a dot,
	\[\frac{1}{k_P-1}\mathbbm{1}_{E_P}\cdot \si = S_{\si}(G)= \Mod_{1,\si}(\Ga).\] Hence, $\widehat{\ga}$ must be equal to $\frac{1}{k_P-1}\mathbbm{1}_{E_P}$. Therefore, $\widehat{\Ga} \subset \Phi$.
\end{proof}

\subsection{The Fulkerson blocker family of the spanning tree family} \label{sec:secondhalf}
Vertex-biconnected graphs have been defined in Definition \ref{def:vertex-biconnected}. Now, we define edge-biconnected graphs.
\begin{definition}  A graph $G=(V,E)$ is said to be {\it edge-biconnected}, if it has at least two vertices, is connected and the removal of any edge does not disconnect the graph.
\end{definition}
It follows that a connected graph with at least two vertices is not vertex-biconnected if and only if there is a vertex such that its removal along with the removal of all adjacent edges disconnects the graph. Such a vertex is called an {\it articulation point}.

Moreover, if a graph $G= (V,E)$ with $|V| \geq 3$ is vertex-biconnected, then it is edge-biconnected, see  \cite[Section III.2]{modern}. When $|V|=2$, the path graph $P_2$ is vertex-biconnected, but not edge-biconnected. Vertex-biconnected graphs also have several equivalent characterizations, one of them is described in the following theorem. 

\begin{theorem}[Section III.2, \cite{modern}] \label{bicon}Let $G=(V,E)$ be a graph with $|V| \geq 3$. Then, $G$ is vertex-biconnected if and only if for any two vertices $u$ and $v$ of $G$, there are two  simple paths connecting $u$ and $v$, with no vertices in common except $u$ and $v$.
\end{theorem}
\begin{corollary}\label{cycle}
	Let $G=(V,E)$ be a graph with $|V| \geq 3$. Then, $G$ is vertex-biconnected if and only if, for any two vertices $u$ and $v$ of $G$, there is a simple cycle containing $u$ and $v$.
\end{corollary}

The following lemma  is stated in \cite{chopraon} without proof. For completeness, we give a proof for this lemma.
\begin{lemma}
	\label{lemma} Let $G=(V,E)$ be a vertex-biconnected graph with $|V| \geq 2$ and $|E| \geq 2$. For any two distinct edges $e_1$ and $e_2$ $\in E$, there exists two spanning trees $\ga_1$ and $\ga_2$ of $G$ such that
	\begin{equation}\label{lemma11}
		\ga_1 \backslash \ga_2 =  \lbrace e_1 \rbrace \qquad  \text{ and} \qquad  \ga_2 \backslash \ga_1 =  \lbrace e_2 \rbrace.
	\end{equation} 
\end{lemma}

\begin{proof}
	First, if $|V| = 2$, then we choose $\ga_1 = \lbrace e_1 \rbrace$ and $\ga_2 = \lbrace e_2 \rbrace$ and (\ref{lemma11}) holds. 
	Next, assume that $|V| \geq 3$, then $G$ is edge-biconnected by Theorem \ref{bicon}. We will show that there exists a cycle in the graph $G$ which contains the edges $e_1$ and $e_2$. Add two new vertices $v_1$ and $v_2$ to the middle of the edges $e_1$ and $e_2$ to create a new graph $G_1$. The removal $v_1$ or $v_2$ from $G_1$ is equivalent to the removal $e_1$ or $e_2$ from $G$ which does not disconnect $G$, hence it does not disconnect $G_1$. The removal of any other vertex from $G_1$ also does not disconnect $G_1$ because $G$ is vertex-biconnected. Therefore, $G_1$ is vertex-biconnected. By Corollary \ref{cycle}, there exists a simple cycle  in $G_1$ that contains $v_1$ and $v_2$, thus there exists a cycle $C$  in $G$ containing $e_1$ and $e_2$.

	Break the cycle $C$ by eliminating the edge $e_2$ from $C$ to create a path $C\backslash e_2$ in $G$. Now proceed as one would in Kruskal's algorithm, successively adding  edges one at a time to $C\backslash e_2$ without creating a cycle. After a number of edge additions, a spanning tree $\ga_1$ of $G$ will result and  it contains the path $C\backslash e_2$. Then, $\ga_1$ contains $e_1$. Now, eliminate $e_1$ from $\ga_1$ and add $e_2$ to $\ga_1$, we obtain another spanning tree $\ga_2$.
	Therefore, (\ref{lemma11}) holds and the proof is completed.
\end{proof}

\begin{claim}\label{claim1}
	We have  $\Theta \subset \widehat{\Gamma}$.
\end{claim}

\begin{proof}
	
	We have $\Theta \subset \Adm(\Gamma)$. Let $ P \in \Theta$, we want to show that 
	\[w :=\frac{1}{k_P-1}\mathbbm{1}_{E_P} \in \Ext(\Adm(\Ga))=\widehat{\Gamma}.\]
	Assume that there are two densities $\rho_1,\rho_2 \in \Adm(\Gamma)$ such that 
	\[ w = \frac{1}{2}(\rho_1+\rho_2).\]
	For every edge $e \in E\backslash E_P$, we have \[\frac{1}{2}(\rho_1(e)+\rho_2(e))=w(e)= 0 \Rightarrow \rho_1(e)=\rho_2(e)=0.\] 
	
	First, assume that $|E_P| \geq 2$. Let $e_1$ and $e_2$ be  two arbitrary distinct edges in $E_P$. By Lemma \ref{lemma}, there exists two spanning trees $\ga_1,\ga_2$ of $G_P$ such that $ \ga_1 \backslash \ga_2 = \lbrace e_1\rbrace$ and $ \ga_2 \backslash \ga_1 =\lbrace e_2\rbrace$. For each $i \in \lbrace 1,2,\dots,k_P \rbrace$, let $G(V_i)$ be the connected subgraph induced by $V_i$. Let $t_i$ be a spanning tree of $G(V_i)$. Let $\ga'_1$ be the union of $t_1,t_2,\dots,t_{k_P},\ga_1$. Let $\ga'_2$  be the union of $t_1,t_2,\dots,t_{k_P},\ga_2$. Then, $\ga'_1$ and $\ga'_2$ are two spanning trees  of $G$ and their restrictions onto $E_P$ are  $\ga_1$ and $\ga_2$, respectively. 
	For any given density $\rho \in \R^E_{\geq 0 }$ and $A \subset E$, denote $\rho(A) := \sum\limits_{e\in A}\rho(e)$.
	Since  $\rho_1,\rho_2 \in \Adm(\Gamma)$, we have
	\[\rho_1(\ga'_1),\rho_1(\ga'_2) ,\rho_2(\ga'_1),\rho_2(\ga'_2) \geq 1.\] 
	Since $\rho_1+\rho_2 = 2w$, we have  $\rho_1(\ga'_1)+\rho_2(\ga'_1) = 2w(\ga'_1) =  2w(\ga_1) = 2$ and  $\rho_1(\ga'_2)+\rho_2(\ga'_2) = 2w(\ga'_2) =2w(\ga_2) =  2$. This implies 
	\[\rho_1(\ga'_1)=\rho_2(\ga'_1)=\rho_1(\ga'_2)=\rho_2(\ga'_2)=1.\]
	Since  $\rho_1(e)=\rho_2(e)=0 $ for all $e \in E\backslash E_P$, we obtain  
	\[\rho_1(\ga_1)=\rho_2(\ga_1)=\rho_1(\ga_2)=\rho_2(\ga_2)=1.\]
	Thus, 
	\begin{align*}
		\rho_1(e_1)+\rho_1( \ga_1 \cap \ga_2)= \rho_1(e_2)+\rho_1( \ga_1 \cap \ga_2)= 1, \\
		\rho_2(e_1)+\rho_2( \ga_1 \cap \ga_2)=\rho_2(e_2)+\rho_2( \ga_1 \cap \ga_2) =1.
	\end{align*}
	Hence, $\rho_1(e_1)=\rho_1(e_2)$ and $\rho_2(e_1)=\rho_2(e_2)$. Thus, since $e_1$ and $e_2$ were arbitrary, $\rho_1$, $\rho_2$ are constant in $E_P$. Since $\rho_1(\ga_1) =\rho_2(\ga_1) =1$ and $\ga_1$ has $k_P-1$ edges, we obtain that 
	
	\[\rho_1(e)= \rho_2(e) = \frac{1}{k_P-1}=w(e), \quad \forall e \in E_P.\]
	Therefore, $\rho_1=\rho_2=w $, hence $w$ is an extreme point of $\Adm(\Ga)$.
	
	If $|E_P|= 1$, then $k_P =2$ since $G_P$ is connected. Suppose $E_P =\lbrace e \rbrace$. Let $\ga$ be a spanning tree of $G$, then it must contain $e$. By the same argument above, we have $\rho_1(e) = \rho_2(e) = \frac{1}{k_P-1} = 1$. Therefore,  $\rho_1=\rho_2=w $, and again $w$ is an extreme point of $\Adm(\Ga)$.
	
	In conclusion, $w \in \widehat{\Ga}$.
\end{proof}

\begin{claim} \label{claim2}
	We have $\Phi \cap \widehat{\Ga} \subset  \Theta.$
\end{claim}
\begin{proof}
	Let $P \in \Phi \cap \widehat{\Ga}$. Suppose $P= \left\{ V_1,V_2,...,V_{k_P} \right\}$, each $V_i$ shrinks to a vertex denoted by $v_i$ in  $G_P$. If $G_P$ is vertex-biconnected, then $P\in \Theta$. If  $G_P$ is not vertex-biconnected. Then, there is an articulation point $v_j$ such that the removal of $v_j$ along with the removal of all adjacent edges of $v_j$ disconnect $G_P$ into two subgraphs $H_1$ and $H_2$ of $G_P$, each of them can have one or several connected components. Denote $A := \lbrace i: v_i \in \text{the node set of } H_1 \rbrace$. Denote $B := \lbrace i: v_i \in \text {the node set of } H_2 \rbrace$. Then,  $A,B$ and $\lbrace j\rbrace$ are disjoint and $A \cup B \cup \lbrace j\rbrace  = \lbrace 1,2\dots,k_p \rbrace$.
	Let \[  \displaystyle M := \left( \bigcup_{i\in B \cup \left\{ j\right\} } V_i \right)\subset V(G),\] we define a partition $P_1$ of the node set $V$ as follows: 
	\[  P_1 := \left\{ V_i: i \in A \right\}  \cup   \left\{ M \right\}.\] 
	Let \[  \displaystyle N := \left( \bigcup_{i\in A \cup \left\{ j\right\} } V_i  \right)\subset V(G) ,\]  we define a partition $P_2$ of the node set $V$ as follows: 
	\[  P_2 := \left\{ V_i: i \in B \right\}  \cup   \left\{ N \right\}. \]
	
	Since $v_j$ is connected with all the connected components of $H_1$ within $G_P \setminus H_2$, then the subgraph of $G_P$ induced by $A \cup \lbrace j \rbrace$  is connected. Therefore, the subgraph of $G$ induced by $N$ is connected. Hence, the partition $P_2$ is feasible. Similarly, the partition $P_1$ is feasible.   We have $k_{P_1} = |A|+1$ and $k_{P_2} = |B|+1$. Thus  $k_{P_1}+k_{P_2}=  |A|+1+|B|+1 =k_P +1$.  We also have that $E_{P_1} \cup E_{P_2} =E_{P}$ and $E_{P_1} \cap E_{P_2} = \emptyset$. Therefore,

	\begin{equation}\label{convex}
		\frac{1}{k_P-1}\mathbbm{1}_{E_P}= \frac{k_{P_1}-1}{k_{P}-1}.\frac{1}{k_{P_1}-1}\mathbbm{1}_{E_{P_1}}+ \frac{k_{P_2}-1}{k_{P}-1}.\frac{1}{k_{P_2}-1}\mathbbm{1}_{E_{P_2}}.
	\end{equation}
	Notice that $\frac{k_{P_1}-1}{k_{P}-1}+\frac{k_{P_2}-1}{k_{P}-1}=1$. So $\frac{1}{k_P-1}\mathbbm{1}_{E_P}$ is a convex combination of  $\frac{1}{k_{P_1}-1}\mathbbm{1}_{E_{P_1}}$ and  $\frac{1}{k_{P_2}-1}\mathbbm{1}_{E_{P_2}}$. Since $\frac{1}{k_{P_1}-1}\mathbbm{1}_{E_{P_1}},\frac{1}{k_{P_2}-1}\mathbbm{1}_{E_{P_2}} \in \Adm(\Ga)$ and  $\frac{1}{k_{P_1}-1}\mathbbm{1}_{E_{P_1}} \neq \frac{1}{k_{P_2}-1}\mathbbm{1}_{E_{P_2}}$, we have  $\frac{1}{k_P-1}\mathbbm{1}_{E_P} \notin \widehat{\Gamma}$, this is a  contradiction. Therefore, $G_P$ is vertex-biconnected and $P\in \Theta$.
\end{proof}

Using Claim \ref{claim:1}, Claim \ref{claim1}, and Claim \ref{claim2}, we complete the proof of Theorem \ref{theo:fulkerson}.

\section{Spanning tree modulus for weighted graphs}\label{sec:weighted2.2}

Let $G= (V,E,\si)$ be a weighted graph with edge weights $\si \in \R^E_{>0}$. Let $\Ga = \Ga_G$ be the family of spanning trees of $G$ with usage vectors given by the indicator functions. Let $\widetilde{\Ga}$ be a Fulkerson dual family of $\Ga$. Let $\rho^*$ and $\eta^*$ be the unique extremal densities for $\Mod_{2,\si}(\Ga)$ and $\Mod_{2,\si^{-1}}(\widetilde{\Ga})$ respectively. 

Note that if $\si \in \Z_{>0}^E$, then we can transform G into an unweighted multigraph by treating the weights $\si(e)$ as edge multiplicities. Then, by continuity and $1$-homogeneity of modulus, it is then straightforward to generalize our results to weighted multigraphs with weights $\si \in \R^E_{>0}$.

We introduce the {\it minimum expected weighted overlap} problem. Let $\cP(\Ga_G)$ be the set of all probability mass functions (or pmf) on $\Ga_G$. Given a pmf $\mu \in \cP(\Ga_G)$, let $\underline{\ga}$ and $\underline{\ga'}$ be two independent random spanning trees, identically distributed by the law $\mu$. We measure the weighted overlap between $\underline{\ga}$ and $\underline{\ga'}$,
\begin{equation}
	\si^{-1}(\underline{\ga} \cap \underline{\ga'}) := \sum\limits_{e \in \underline{\ga} \cap \underline{\ga'}} \si^{-1}(e),
\end{equation}
which is a random variable whose expectation is denoted by  $\bE_{\mu} \left[ \si^{-1}(\underline{\ga} \cap \underline{\ga'}) \right]$. Then, the {\it minimum expected weighted overlap} ($\MEO_{\si^{-1}}$) problem is the following problem: 

\begin{equation} \label{meoweighted}
	\begin{array}{ll}
		\text{minimize}    &\bE_{\mu} \left[ \si^{-1}(\underline{\ga} \cap \underline{\ga'}) \right] \\
		\text{subject to } & \mu \in \cP(\Ga_G). 	 
	\end{array}
\end{equation}
Next, we give a theorem characterize the relation between spanning tree modulus and the ($\MEO_{\si^{-1}}$) problem.
\begin{theorem} Let $G= (V,E,\si)$ be a weighted connected graph with edge weights $\si \in \R^E_{>0}$. Let $\Ga = \Ga_G$ be the spanning tree family, and let $\widetilde{\Ga}$ be a Fulkerson dual family. Then, $\rho \in  \R^E_{\geq 0}$, $\eta \in  \R^E_{\geq 0}$ and $\mu \in \bP(\Ga)$ are optimal respectively for $\Mod_{2,\si}(\Ga)$, $\Mod_{2,\si^{-1}}(\widetilde{\Ga})$ and $\MEO_{\si^{-1}}(\Ga)$ if and only if the following conditions are satisfied.
	
	\begin{align*}
		{(i)} &\qquad \rho \in \Adm(\Ga), \qquad \eta = \cN^T\mu,\\
		{(ii)} &\qquad \eta(e) = \frac{\si(e)\rho(e)}{\Mod_{2,\si}(\Ga)} \qquad \forall e \in E,\\
		{(iii)} &\qquad \mu(\ga)(1-\ell_\rho(\ga)) = 0 \qquad \forall \ga \in \Ga.
	\end{align*}
	In particular, 
	\begin{align}
		\MEO_{\si^{-1}}(\Ga) = \Mod_{2,\si}(\Ga)^{-1} = \Mod_{2,\si^{-1}}(\widetilde{\Ga}).
	\end{align}
\end{theorem}

\begin{definition}\label{def:homogeneous-weighted}
	Let $G= (V,E,\si)$ be a weighted graph with edge weights $\si \in \R^E_{>0}$. Let $\Ga = \Ga_G$ be the family of spanning trees of $G$ with usage vectors given by the indicator functions. Let $\widetilde{\Ga}$ be a Fulkerson dual family of $\Ga$. Let $\rho^*$ and $\eta^*$ be the unique optimal densities for $\Mod_{2,\si}(\Ga)$ and $\Mod_{2,\si^{-1}}(\widetilde{\Ga})$ respectively. Then, the graph $G$ is said to be {\it homogeneous} if $\si^{-1}\eta^*$ is constant, or equivalently, $\rho^*$ is constant.
\end{definition}


\begin{theorem}\label{whom}
	Let $G= (V,E,\si)$ be a weighted connected graph with edge weights $\si \in \R^E_{>0}$. Let $\Ga = \Ga_G$ be the family of spanning trees of $G$. Let $\widetilde{\Ga}$ be a Fulkerson dual family of $\Ga$.  Define the density $n_{\si}$:
	
	\begin{equation}\label{nhom}
		n_{\si}(e) :=\frac{\si(e)}{\si(E)}(|V|-1) \qquad \forall e \in E.
	\end{equation} 
	Then, $G$ is homogeneous if and only if  $\eta_{\si} \in \Adm(\widetilde{\Ga})$.
\end{theorem}

 We want to give a new characterization of a homogeneous graph. We start with the following lemma.

\begin{lemma}\label{theo:chom}
	Let $G= (V,E,\si)$ be a weighted graph with edge weights $\si \in \R^E_{>0}$. Let $\Ga = \Ga_G$ be the family of spanning trees of $G$. Let $\widetilde{\Ga}$ be a Fulkerson dual family of $\Ga$.  Let $n_{\si}$ be defined as in (\ref{nhom}). Then,  $G$ is homogeneous if and only if  $\eta_{\si} \in \co(\Ga).$
\end{lemma}

\begin{proof}
	We recall that \[\Adm(\widetilde{\Ga}) = \Dom(\Ga)= \co(\Ga)+ \R^E_{\geq 0 }.\]  
	By Theorem \ref{whom}, $G$ is homogeneous if and only if  $\eta_{\si} \in \Adm(\widetilde{\Ga}).$ Assume that $\eta_{\si} \in \Adm(\widetilde{\Ga}).$ Then, $\eta_{\si} = x+z$ where $x \in \co(\Ga)$ and $z\in \R^E_{\geq 0 }$. Since $x \in \co(\Ga)$, then $x = \sum\limits_{\ga \in \Ga} u(\ga)\cN(\ga,\cdot)$ for some $\mu \in \R^{\Ga}$ satifying $\mu \geq 0$ and $\sum\limits_{\ga\in \Ga} \mu(\ga) = 1$. Then,
	
	\begin{align*}
		x(E)=\sum\limits_{e \in E}\sum\limits_{\ga \in \Ga} \mu(\ga)\cN(\ga,e)
		=\sum\limits_{\ga \in \Ga}\sum\limits_{e \in E} \mu(\ga)\cN(\ga,e)= |V|-1.
	\end{align*}
	Notice that \[\eta_{\si}(E) =  \sum\limits_{e\in E}\eta_{\si}(e)= |V|-1.\]
	Therefore, \[z(E) = \eta_{\si}(E) - x(E) =  0.\]
	Hence, $z = 0$ and $\eta_{\si} = x \in \co(\Ga).$
	
	Assume that $\eta_{\si} \in \co(\Ga)$, then  
	$\eta_{\si} \in \Adm(\widetilde{\Ga})$ because $\co(\Ga) \subset \Adm(\widetilde{\Ga})$.
	In conclusion, $G$ is homogeneous if and only if  $\eta_{\si} \in \co(\Ga).$
\end{proof} 

\begin{theorem}\label{homo3}
	Let $G= (V,E,\si)$ be a weighted graph with edge weights $\si \in \R^E_{>0}$. Let $\Ga = \Ga_G$ be the family of spanning trees of $G$. Let $\coni(\Ga)$ be the conical hull of $\Ga$. Then, $G$ is homogeneous if and only if $\si \in \coni (\Ga)$. 
\end{theorem}
\begin{proof}
	By Lemma \ref{theo:chom}, it is enough to show that $\eta_{\si} \in \co(\Ga)$ if and only if $\si \in \coni (\Ga)$.
	
	Assume that $\eta_{\si} \in \co(\Ga)$. Then $\eta_{\si}$ is a convex combination of elements in $\Ga$. That implies \[\si = \frac{\si(E)}{|V|-1}\eta_{\si}\] is a conical combination of of elements in $\Ga$. In other words, $\si \in \coni (\Ga)$. 
	
	Assume that $\si \in \coni (\Ga)$. Then, $\si = \sum\limits_{\ga \in \Ga} u(\ga)\cN(\ga,\cdot)$ with $u(\ga) \geq 0 \quad \forall \ga \in \Ga$. Then, \[\si(E) = \sum\limits_{e \in E}\sum\limits_{\ga \in \Ga} \mu(\ga)\cN(\ga,e)
	=\sum\limits_{\ga \in \Ga}\sum\limits_{e \in E} \mu(\ga)\cN(\ga,e) =u(\Ga)(|V|-1)\]  where $u(\Ga) := \sum\limits_{\ga \in \Ga}u(\ga)$. Therefore, 
	\[\eta_{\si}(e) =\frac{\si(e)}{\si(E)}(|V|-1) =  \frac{\si(e)}{u(\Ga)} = \frac{1}{u(\Ga)}\sum\limits_{\ga \in \Ga} u(\ga)\cN(\ga,e).\] In other words, $\eta_{\si} \in \co(\Ga)$.
	In conclusion, $G$ is homogeneous if and only if $\si \in \coni (\Ga)$. 
\end{proof}

Finally, we generalize our results to weighted multigraphs.
\begin{theorem} Let $G= (V,E,\si)$ be a weighted connected graph with edge weights $\si \in \R^E_{>0}$. Let $S_{\si}(G)$ be the strength of $G$. Let $\Ga = \Ga_G$ be the family of spanning trees of $G$. Let $\widetilde{\Ga}$ be a Fulkerson dual family of $\Ga$.
	Let $\eta^*$  be the optimal density for $\Mod_{2,\si^{-1}}(\widetilde{\Ga})$. Denote
	\begin{equation}
		E_{ max} :=\left\{ e\in E: \si^{-1}(e)\eta^*(e)=\max\limits_{e \in E}\si^{-1}(e)\eta^*(e) =: (\si^{-1}\eta^*)_{  max} \right\}.
	\end{equation}
	Then, there exists a critical partition $P_{ max}$ for the strength problem such that $E_{P_{ max}} = E_{ max}$ and
	\begin{equation}
		(\si^{-1}\eta^*)_{  max} = \frac{1}{S_{\si}(G)}.
	\end{equation}
	
\end{theorem}

\begin{theorem}\label{emin}
	Let $G= (V,E,\si)$ be a weighted connected graph with edge weights $\si \in \R^E_{>0}$.  Let $D_{\si}(G)$ be the maximum denseness of $G$. Let $\Ga = \Ga_G$ be the family of spanning trees of $G$. Let $\widetilde{\Ga}$ be a Fulkerson dual family of $\Ga$.
	Let $\eta^*$  be the optimal density for $\Mod_{2,\si^{-1}}(\widetilde{\Ga})$.  Denote 
	\begin{equation}
		E_{ min} :=\left\{ e\in E: \si^{-1}(e)\eta^*(e)=\min\limits_{e \in E}\si^{-1}(e)\eta^*(e) =: (\si^{-1}\eta^*)_{  min} \right\}.
	\end{equation}
	Let $H_{ min}$ be the edge-induced subgraph induced by $E_{ min}$. Then, every connected component $H$ of $H_{ min}$ is a maximum denseness subgraph of $G$  and
	
	\begin{equation}
		(\si^{-1}\eta^*)_{  min} = \frac{1}{D_{\si}(G)}.
	\end{equation}

\end{theorem}
\begin{theorem}
	Let $G= (V,E,\si)$ be a weighted connected graph with edge weights $\si \in \R^E_{>0}$. Let $S_{\si}(G)$ be the strength of $G$. Let $D_{\si}(G)$ be the maximum denseness of $G$. Let $\Ga = \Ga_G$ be the family of spanning trees of $G$. Let $\widetilde{\Ga}$ be a Fulkerson dual family of $\Ga$.
	Let $\eta^*$  be the optimal density for $\Mod_{2,\si^{-1}}(\widetilde{\Ga})$. Then,
	\begin{equation}\label{maxmin-weighted}
		\frac{1}{(\si^{-1}\eta^*)_{ max}} = S_{\si}(G)\leq \theta_{\si}(G)\leq D_{\si}(G) = \frac{1}{(\si^{-1}\eta^*)_{ min}}.
	\end{equation}
	Moreover, the graph $G$ is homogeneous if and only if $S_{\si}(G) = D_{\si}(G)$.
\end{theorem}

\bibliographystyle{acm}
\bibliography{paper}

\begin{thebibliography}{10}

\bibitem{modulus}
{\sc Albin, N., Brunner, M., Perez, R., Poggi-Corradini, P., and Wiens, N.}
\newblock Modulus on graphs as a generalization of standard graph theoretic
  quantities.
\newblock {\em Conform. Geom. Dyn. 19\/} (2015), 298--317.

\bibitem{pietroblocking}
{\sc Albin, N., Clemens, J., Fernando, N., and Poggi-Corradini, P.}
\newblock Blocking duality for {$p$}-modulus on networks and applications.
\newblock {\em Ann. Mat. Pura Appl. (4) 198}, 3 (2019), 973--999.

\bibitem{pietrofairest}
{\sc Albin, N., Clemens, J., Hoare, D., Poggi-Corradini, P., Sit, B., and
  Tymochko, S.}
\newblock Fairest edge usage and minimum expected overlap for random spanning
  trees.
\newblock {\em Discrete Math. 344}, 5 (2021), Paper No. 112282, 24.

\bibitem{polynomial}
{\sc Albin, N., Kottegoda, K., and Poggi-Corradini, P.}
\newblock A polynomial-time algorithm for spanning tree modulus.
\newblock {\em arXiv: Combinatorics\/} (2020).

\bibitem{pietrosecure}
{\sc Albin, N., Kottegoda, K., and Poggi-Corradini, P.}
\newblock Spanning tree modulus for secure broadcast games.
\newblock {\em Networks 76}, 3 (2020), 350--365.

\bibitem{pietrominimal}
{\sc Albin, N., and Poggi-Corradini, P.}
\newblock Minimal subfamilies and the probabilistic interpretation for modulus
  on graphs.
\newblock {\em J. Anal. 24}, 2 (2016), 183--208.

\bibitem{ba2023}
{\sc Baiou, M., and Barahona, F.}
\newblock On some algorithmic aspects of hypergraphic matroids.
\newblock {\em Discrete Mathematics 346}, 2 (2023), 113222.

\bibitem{bertsimas}
{\sc Bertsimas, D., and Tsitsiklis, J.~N.}
\newblock {\em Introduction to linear optimization}, vol.~6.
\newblock Athena scientific Belmont, MA, 1997.

\bibitem{modern}
{\sc Bollob\'{a}s, B.}
\newblock {\em Modern graph theory}, vol.~184 of {\em Graduate Texts in
  Mathematics}.
\newblock Springer-Verlag, New York, 1998.

\bibitem{chopraon}
{\sc Chopra, S.}
\newblock On the spanning tree polyhedron.
\newblock {\em Oper. Res. Lett. 8}, 1 (1989), 25--29.

\bibitem{frank2003decomposing}
{\sc Frank, A., Kir{\'a}ly, T., and Kriesell, M.}
\newblock On decomposing a hypergraph into k connected sub-hypergraphs.
\newblock {\em Discrete Applied Mathematics 131}, 2 (2003), 373--383.

\bibitem{fulkersonblocking}
{\sc Fulkerson, D.~R.}
\newblock Blocking polyhedra.
\newblock In {\em Graph {T}heory and its {A}pplications ({P}roc. {A}dvanced
  {S}em., {M}ath. {R}esearch {C}enter, {U}niv. of {W}isconsin, {M}adison,
  {W}is., 1969)\/} (1970), Academic Press, New York, pp.~93--112.

\bibitem{fulkersonanti}
{\sc Fulkerson, D.~R.}
\newblock Blocking and anti-blocking pairs of polyhedra.
\newblock {\em Math. Programming 1\/} (1971), 168--194.

\bibitem{gusfieldtree}
{\sc Gusfield, D.}
\newblock Connectivity and edge-disjoint spanning trees.
\newblock {\em Inform. Process. Lett. 16}, 2 (1983), 87--89.

\bibitem{edge-disjoint}
{\sc Nash-Williams, C. S. J.~A.}
\newblock Edge-disjoint spanning trees of finite graphs.
\newblock {\em J. London Math. Soc. 36\/} (1961), 445--450.

\bibitem{on}
{\sc Tutte, W.~T.}
\newblock On the problem of decomposing a graph into {$n$} connected factors.
\newblock {\em J. London Math. Soc. 36\/} (1961), 221--230.

\end{thebibliography}
\end{document}